\documentclass[twoside,11pt]{article}

\usepackage{jmlr2e}

\usepackage{algorithm}  
\usepackage{algpseudocode}  
\usepackage{amsmath}
\usepackage{bbm}
\usepackage{multirow}
\usepackage{chngpage}
\usepackage{verbatim}
\usepackage{mathtools}
\usepackage{booktabs}
\usepackage{float}
\usepackage{array}
\newcolumntype{P}[1]{>{\centering\arraybackslash}p{#1}}
\usepackage{tabularx}
\usepackage{enumitem}
\usepackage{caption}
\usepackage{multicol}
\newtheorem{assumption}{Assumption}
\usepackage{xcolor}
\usepackage{hyperref}
\hypersetup{
    colorlinks=true,
    linkcolor=blue,
    filecolor=magenta,      
    urlcolor=cyan,
    citecolor={blue}
}

\ShortHeadings{Escaping Saddle Points with SCSG Methods}{Liang, Tong, Zhu and Bi}
\firstpageno{1}

\begin{document}

\title{Escaping Saddle Points with Stochastically Controlled Stochastic Gradient Methods}

\author{\name Guannan Liang \email guannan.liang@uconn.edu
       \AND
       \name Qianqian Tong \email qianqian.tong@uconn.edu
       \AND
       \name Chunjiang Zhu \email chunjiang.zhu@uconn.edu
       \AND
       \name Jinbo Bi \email jinbo.bi@uconn.edu\\
       \addr University of Connnecticut\\ Storrs, CT 06269, USA}
\editor{ }
\maketitle

\begin{abstract}
	Stochastically controlled stochastic gradient (SCSG) methods have been proved to converge efficiently to first-order stationary points which, however, can be saddle points in nonconvex optimization. It has been observed that a stochastic gradient descent (SGD) step introduces anistropic noise around saddle points for deep learning and non-convex half space learning problems, which indicates that SGD satisfies the correlated negative curvature (CNC) condition for these problems. {   Therefore, we propose to use a separate SGD step to help the SCSG method escape from strict saddle points, resulting in the CNC-SCSG method. The SGD step plays a role similar to noise injection but is more stable. We prove that the resultant algorithm converges to a second-order stationary point with a convergence rate of $\tilde{O}( \epsilon^{-2} log( 1/\epsilon))$ where $\epsilon$ is the pre-specified error tolerance. This convergence rate is independent of the problem dimension, and is faster than that of CNC-SGD. 
	A more general framework is further designed to incorporate the proposed CNC-SCSG into any first-order method for the method to escape saddle points. Simulation studies illustrate that the proposed algorithm can escape saddle points in much fewer epochs than the gradient descent methods perturbed by either noise injection or a SGD step. }
\end{abstract}

\section{Introduction}

In this paper, the problem of interest is the nonconvex finite-sum problem of the form
\begin{equation}
\vspace{-.2cm}\label{finite_sum}
\min_{x \in \mathbb{R}^d} f( x ) := \frac{1}{n} \sum_{z=1}^{n} f_z( x),
\end{equation}
where both $f$ and $f_z$ ($z \in [n] :=\{1,2,\dots,n\}$) may be nonconvex, and their gradients and Hessians are Lipschitz continuous. Problem (\ref{finite_sum}) plays an essential role in many machine learning methods especially in the training of deep neural networks.  
The stochastic gradient descent (SGD) method and its variants are commonly used  for solving (\ref{finite_sum}). However, finding global or local minimum of $f$ generally is NP-hard \citep{anandkumar2016efficient}. {  Hence, people turn to find sub-optimal solutions for Problem (\ref{finite_sum}), such as first-order stationary points \citep{nesterov2013introductory}  or second-order stationary points \citep{ge2015escaping}.} 

An $\epsilon$-first-order stationary point $x$ can be guaranteed by the gradient descent (GD) method and its variants, such that for an arbitrarily small $\epsilon>0$,
\begin{equation} \label{1st_stationary_condition}
\| \nabla f( x ) \| \leq \epsilon
\end{equation} 
\citep{nesterov2013introductory,ghadimi2016accelerated,carmon2017convex,reddi2016stochastic,lei2017non}. 
Specifically, for a smooth nonconvex $f$, an $\epsilon$-first-order stationary point can be found by the GD in $\mathcal{O}( \epsilon^{-2})$ iterations, or by the SGD in $\mathcal{O}( \epsilon^{-4} )$ iterations \citep{nesterov2013introductory} where $\epsilon$ is the error tolerance chosen beforehand.
Later, the accelerated gradient descent (AGD) method also achieved a convergence rate $\mathcal{O}( \epsilon^{-2})$ \citep{ghadimi2016accelerated}, and the Guarded-AGD algorithm improved over the AGD algorithm and reached a convergence rate $O(\epsilon^{-\frac{7}{4}})$.
However, $\epsilon$-first-order stationary points can be saddle points or even local maximizers in non-convex optimization \citep{jin2017escape}.   

{  Hence, there has been theoretical analysis to find second-order stationary points for Problem (\ref{finite_sum}) \citep{ge2015escaping, jin2017escape,nesterov2006cubic} where an ($\epsilon_g$, $\epsilon_h$)-second-order stationary point $x$ is sought such that:
\begin{equation} \label{2nd_stationary_condition}
\| \nabla f( x ) \| \leq \epsilon_g ~~~\mbox{and}~~~  \nabla^2 f( x )  \succeq -\epsilon_h  I 
\end{equation}  
for some small positive $\epsilon_g$ and $\epsilon_h$, and $I$ is the identity matrix.  If we assume that $f( x )$ has strict saddles (see Definition \ref{def:strictsaddle}), a second-order stationary point is a local minimizer. 
In order to obtain a sub-optimal solution $x$ which satisfies (\ref{2nd_stationary_condition}), some works explicitly calculate the eigenvector direction corresponding to the most negative eigenvalue of the Hessian matrix $\nabla^2 f( x )$ around saddle points, such as the cubic regularized Newton's (CRN) methods \citep{nesterov2006cubic,agarwal2017finding,reddi2018generic,tripuraneni2017stochastic, zhou2018stochastic, curtis2017trust}.The CRN-based methods require an {\em oracle} to compute the search direction at each iteration, where the oracle is difficult to derive and time-consuming to compute. Thus, some other works including this present work develop perturbation-based methods to escape saddle points. 

Some perturbation-based methods rely on the strategy of injecting isotropic noise along the search direction, such as in the perturbed GD (PGD) and perturbed SGD (PSGD) methods \citep{ge2015escaping, jin2017escape,zhang2017hitting}. A recent work observes that the stochasticity in SGD steps induces intrinsic anisotropic noise around saddle points, and hence proposes to utilize this noise by running a separate SGD step to escape saddle points  \citep{daneshmand2018escaping}. The advantage of such a method is that the anisotropic noise does not reduce along the negative curvature around saddle points when the deep neural network (the size of the optimization problem) becomes massive. Thus, this leads to two new methods both based on the assumption called the correlated negative curvature (CNC), namely, the CNC-GD and CNC-SGD methods.

Although the anisotropy of the noise in SGD steps has been extensively studied \citep{zhu2018anisotropic,daneshmand2018escaping,simsekli2019tail}, there are discrepancies in the discussion. \cite{fang2019sharp} assumes dispersive noise structure in SGD, such as the structure of Gaussian noise or uniform ball-shaped noise,  which then supports the explicit uniform ball-shaped noise injection procedure in the PGD method \citep{jin2017escape}. Although this assumption of dispersive (similar to isotropic) noise enables the authors to analyze SGD directly, it deserves a second thought. The variance of  Gaussian noise or uniform ball-shaped noise diminishes to zero along the negative curvature when the problem size increases, but practically the noise induced by SGD steps does not. For example, \cite{daneshmand2018escaping}  experimentally show that  the variance of noise caused by SGD along the negative curvature direction is not influenced by the dimension of deep neural networks. More importantly, \cite{simsekli2019tail} have studied the long tail behavior of noise from SGD in  deep neural networks and shown that the noise in SGD does not follow the behavior of Gaussian noise as well as uniform ball-shaped noise. Hence, we assume in this paper that SGD satisfies the CNC condition which has been theoretically proved for the non-convex half-space learning problems. 

}

  \begin{center}
\vspace{-.8cm}
  \begin{minipage}[t]{.7\textwidth}
\begin{algorithm}[H]
	\caption {Finding a second-order stationary point }
	\label{alg:escaping_saddle} 	
	\begin{algorithmic}[1]		
		\Require  $\epsilon_g$ and $\epsilon_h$
		\For{$k=0,1,2,...$}
		\State $y^{k} = Gradient$-$Focused$-$Optimizer( x^{k-1})$
		\State $x^{k} = Escaping$-$Saddle$-$Algorithm( y^{k})$
		\If{ ($\epsilon_g$, $\epsilon_h$)-second-order conditions}
		\State Return $x^k$
		\EndIf
		\EndFor		
	\end{algorithmic}	
\end{algorithm}
\end{minipage}
\end{center}

For perturbation-based algorithms, we summarize them into a general framework as in Algorithm  \ref{alg:escaping_saddle}.
The procedure $Gradient$-$Focused$-$Optimizer(\cdot)$ in Line 2 can be any algorithm that converges to an $\epsilon$-first-order stationary point, such as, the GD, SGD, and stochastically controlled stochastic gradient (SCSG) methods. The procedure $Escaping$-$Saddle$-$Algorithm(\cdot)$ in Line 3 is for escaping from the region around saddle points (also called the stuck region), typically by taking a negative curvature direction explicitly or implicitly. {   A few prior works attempted to use variance reduction methods in the $Gradient$-$Focused$-$Optimizer(\cdot)$ \citep{xu2017first,allen2017neon2,zhou2018finding} to reach a better convergence rate for second-order stationary points, but they have not been used in the $Escaping$-$Saddle$-$Algorithm(\cdot)$.}{ We are the first work to use variance reduced methods inside $Escaping$-$Saddle$-$Algorithm(\cdot)$ .}

\textbf{Our contributions.} In this paper, we first propose a variance reduced algorithm, which we call CNC-SCSG, to escape strict saddle points by incorporating SGD into SCSG. We show that it converges to a second-order stationary point with a convergence rate faster than that of the CNC-SGD and comparable to that of the CNC-GD, under the assumption of the CNC condition. The computational complexity in terms of the number of IFO calls (see Definition 3) has also been analyzed. Comparing with GD methods, perturbed by either Gaussian noise or one SGD step, our simulation studies show that the CNC-SCSG converges in much fewer epochs than the perturbed GD methods even though they theoretically have comparable convergence rate.
Furthermore, we generalize the CNC-SCSG into a framework that separates the part for escaping saddle points from the main algorithm so that it can be combined with any existing first-order algorithm, providing a mechanism that allows further search for faster convergence. 

\section{Related works}

We briefly review the existing methods for finding second-order stationary points in the following three lines. 

{\bf Perturbation methods.}
A line of research finds second-order stationary points by injecting perturbation when the algorithm gets close to saddle points. The perturbation can be either an explicit isotropic noise injected to the current gradient or search direction \citep{ge2015escaping, jin2017escape,jin2018accelerated,xu2017first, allen2017neon2,allen2017natasha,jin2019stochastic,fang2019sharp,ge2019stabilized,li2019ssrgd},
or a  SGD step introducing intrinsic and anisotropic noise implicitly \citep{daneshmand2018escaping}.

As in the PSGD \citep{ge2015escaping,jin2019stochastic,fang2019sharp}, PGD \citep{jin2017escape} and perturbed AGD (PAGD) \citep{jin2018accelerated} methods , researchers have shown that the stuck region is a ``thin pancake". Thus injecting random noise to gradients can help the algorithm to jump out of the stuck region with high probability. However, because the noise is isotropic, this strategy is inefficient in high dimension. The Neon \citep{xu2017first} and Neon2 \citep{allen2017neon2} methods connected the noise perturbed methods PGD and PSGD to the search methods of negative curvature direction, such as the power method \citep{kuczynski1992estimating} and Oja's method \citep{oja1982simplified}. These methods for the first time successfully separated the $Escaping$-$Saddle$-$Algorithm(\cdot)$ from the main GD process. Because of the separation, the frameworks are easy to incorporate any suitable algorithm for  $Gradient$-$Focused$-$Optimizer(\cdot)$. The overall convergence rate has less dependence on the dimension $d$, i.e. from $poly(d)$ in PSGD  \citep{ge2015escaping} to $poly( log(  d ))$\citep{xu2017first,allen2017neon2}. Nevertheless, as the convergence rate still depends on $d$, it may cause issues in deep learning models with massive network parameters. Recently, \citep{daneshmand2018escaping} observed that SGD works very well in deep learning applications, and experimentally showed that the anisotropic noise created by SGD has a high variance along the negative curvature direction. The variance does not diminish with the increase in the size of deep neural networks. This observation implies that SGD satisfies the CNC condition (see Definition \ref{def:CNC-Condition}). Hence, a single step of SGD can be used to escape saddle points, which can also remove the dependency on problem dimension as well. 

{ \bf Cubic regularized Newton's (CRN) methods.}
Another line of studies for escaping saddle points is based on the CRN method \citep{nesterov2006cubic} and its variants \citep{agarwal2017finding,reddi2018generic,tripuraneni2017stochastic, zhou2018stochastic, curtis2017trust}.
These methods use second-order information directly. The CRN method \citep{nesterov2006cubic} can converge to an ($\epsilon, \sqrt{\epsilon})$-second-order stationary point in $O(\epsilon^{-\frac{3}{2}})$ steps, if there is an oracle which can return an update direction $h$ at each iteration to minimize the cubic function $m(h) = g^Th + \frac{1}{2}h^THh + \frac{\rho}{6}\| h \|^3$, where $g=\nabla f( x ) $, $H = \nabla^2 f( x )$ and $\rho$ is the second-order smoothness parameter of $f$. The FastCubic method in \citep{agarwal2017finding} improved the CRN method and achieved a better convergence rate. Its stochastic version has also been proposed by computing stochastic gradient and Hessian using two mini-batches respectively at each iteration \citep{tripuraneni2017stochastic}. The SVRCubic method \citep{zhou2018stochastic} introduced the variance reduction technique to the stochastic CRN method to reach the  convergence rate of its deterministic version. 

 \textbf{Stochastic variance reduction methods.} For convex optimization, variance reduction methods have been extensively studied, such as, the stochastic variance reduced gradient (SVRG) \citep{johnson2013accelerating}, stochastically controlled stochastic gradient (SCSG) \citep{lei2016less}, and stochastic average gradient (SAGA) \citep{defazio2014saga} methods, and they are well-known for faster convergence rates. In non-convex optimization, variance reduction methods have been proved to converge to $\epsilon$-first-order stationary points \citep{reddi2016stochastic,lei2017non}. Some of these methods have been used in the $Gradient$-$Focused$-$Optimizer(\cdot)$ step of Algorithm \ref{alg:escaping_saddle} to find second-order stationary points \citep{xu2017first,allen2017neon2,zhou2018finding}.

\section{Preliminaries}

We use  uppercase letters, e.g. $A$, to denote matrices and lowercase letters, e.g. $x$, to denote vectors. We use $\|\cdot\|$ to denote the 2-norm for vectors and  the spectral norm for matrices, and $\lambda_{min}(\cdot)$ to denote the minimum eigenvalue. For two matrices $A$ and $B$, $A\succeq B$ iff $A-B$ is positive semi-definite. In this paper, the notation $O(\cdot)$ is used to hide constants which do not rely on the setup of the problem parameters and the notation $\tilde{O}(\cdot)$ is used to hide all $\epsilon$-independent constants. The operator $E[\cdot]$ represents taking the expectation over all randomness, $[n]$ denotes the integer set $\{1, ..., n\}$, $\nabla f( \cdot )$, $\nabla f_I( \cdot)$ and $\nabla f_z ( \cdot )$ are the full gradient, the stochastic gradient over a mini-batch $I\subset [n]$ and the stochastic gradient over a single training example indexed by $z \in [n]$, respectively. We also assume that the probability of observing $z$ given the model parameter $x$, $ P( z | x) \propto exp^{ -f_z( x)} $. 
\begin{definition}\label{def:strictsaddle}{(Strict saddle)}
	A twice differential function f has strict saddle if all its local minima have $\nabla^2 f( x ) \succ 0$ and all its other stationary points satisfy that $\lambda_{min}( \nabla^2 f( x )) < 0$.
\end{definition}
\begin{definition}\label{def:CNC-Condition}
	{(CNC-condition)}
	A differentiable  finite-sum function $f(x): \mathbb{R}^d \rightarrow \mathbb{R}$ satisfies the CNC condition if $\exists \tau \geq 0, \text{s.t.}~ E[\langle\mathbf{v_x}, \nabla f_z(x)\rangle^2] \geq \tau,$	 $\forall x \in \{ x: \lambda_{min}(\nabla^2 f(x)) < 0 \}$ and $z\in [n]$, where $\mathbf{v_x}$ is the eigenvector corresponding to the minimum eigenvalue $\lambda_{min}( \nabla^2 f(x ))$ and $E[\cdot]$ is the expectation taken over the random sample $z$.
\end{definition}
\begin{definition}
    {(Incremental First-order Oracle (IFO)\citep{agarwal2014lower})} An IFO  is a subroutine  that takes a point $x\in \mathbb{R}^d$ and an index $z \in [n]$ and returns a pair ($f_z(x)$, $\nabla f_z(x)$) .
\end{definition}
\begin{definition}
	(Geometric distribution) A random variable $N$ is said to follow a geometric distribution $Geom( \gamma)$ for some $\gamma >0$, denoted as $N \sim Geom( \gamma)$, if  $N$ is a non-negative integer and the probability	$P( N = k ) = ( 1 - \gamma) \gamma^k$	for all $ k = 0, 1, \cdots$. It holds that $E[N] = \frac{\gamma}{1-\gamma}$.
\end{definition}
In this paper, we assume that the function $f(x)$ is lower-bounded by a constant $f^*$, which is the optimal value of the objective. As commonly used in the related works on escaping saddle points, we also assume that the gradient and Hessian of $f_z$ are locally stable (Assumption 1 (1)-(2)), and the variance of the stochastic gradients is bounded from above (Assumption 1 (3)).

\begin{assumption}\label{assumption4f}
The twice differentiable function $f_z(x )$, $z \in [n]$, satisfies: 
\begin{enumerate}
\item  $L$-Lipschitz gradient, i.e., $\| \nabla f_z( x_1 ) - \nabla f_z( x_2 ) \| \leq L \| x_1 - x_2 \|$, $\forall x_1, x_2 \in \mathbb{R}^d$;\\
\item  $\rho$-Lipschitz Hessian, i.e., $\| \nabla^2 f_z( x_1 ) - \nabla^2 f_z( x_2 ) \| \leq \rho \| x_1 - x_2 \|$, $\forall x_1, x_2 \in \mathbb{R}^d$; \\
\item  $\exists~l\geq 0$, such that, $\| \nabla f_z( x ) \| \leq l, \forall x \in \mathbb{R}^d$, and $\forall z \in [n]$.
\end{enumerate}
\end{assumption}
We further exploit the observation given in \citep{daneshmand2018escaping} that SGD steps intrinsically induce anisotropic noise and the variance of noise along negative curvature direction does not decrease with the increase of deep neural network dimension size. We emphasize that the original CNC-condition in \citep{daneshmand2018escaping} does not necessarily hold for the entire $x$ domain. It only needs to hold at the points around strict saddle points for SGD to jump out of stuck regions.
\begin{assumption}\label{CNC_assumption}
	In Problem (\ref{finite_sum}), we assume that $f( x )$ satisfies the CNC-condition at points in the stuck regions.
\end{assumption}
The CNC-condition can be regarded as instability of the iteration dynamics in SGD around strict saddle points.
Unlike the methods of injecting isotropic noise to the first order algorithms as in \cite{ge2015escaping, jin2017escape, jin2018accelerated,xu2017first,allen2017neon2} which have a convergence rate dependent on the dimension $d$ (i.e., either $poly( d )$ or $poly \log( d )$), the most important advantage of using the anisotropic noise induced by SGD is that the convergence rate no longer depends on the dimension $d$.  When optimizing deep neural networks with millions of parameters, the anisotropic noise provided by SGD can have a significant advantage over injected isotropic noise. 
\begin{remark}
	Around saddle points, the variance-covariance matrix of the stochastic gradient $\nabla f_z( x)$ is approximately equal to the empirical Fisher's information matrix (see Appendix E). In practice, under the assumptions that the sample size $n$ is large enough and the current $x$ is not too far away from the ground truth parameter, Fisher's information matrix $E_{P(z|x)}[\nabla^2 \log( P( z|x))]$ approximately equals to the Hessian matrix $H$ of $f(\cdot)$ at $x$ (Note that taking logarithm of the likelihood $\prod_{z=1}^{n}P(z|x)$ yields a finite sum $f$). Thus, the intrinsic noise in SGD is naturally embedded into the Hessian information of $f$ \citep{zhu2018anisotropic}. This explains the anisotropic behavior of the noise in SGD.
\end{remark}

\section{The CNC-SCSG algorithm}

In this section, we present a novel algorithm CNC-SCSG (Algorithm \ref{alg:SVRG_SGD}), which utilizes SCSG and SGD to escape strict saddle points and is able to find ($\epsilon$, $\sqrt{\rho}\epsilon^{\frac{2}{5}}$)-second-order stationary points.
The SCSG algorithm is a member of the SVRG family, which was first proposed in \cite{lei2016less} and showed competitive time complexity in convex optimization. Later it was extended to non-convex optimization \citep{lei2017non}. The original SVRG method computes a full gradient before starting an inner loop which typically contains $O(n)$ iterations. The SCSG differs from the SVRG mainly because 
the number of iterations in the inner loop is stochastically determined by an integer sampled from $Geom(\gamma)$ where $\gamma$ is related to the size $b$ of the mini-batch $I$ used in the inner loop.

\begin{center}
\vspace{-1.cm}
\begin{minipage}[t]{.7\textwidth}
\begin{algorithm}[H]
	\caption {CNC-SCSG}
	\label{alg:SVRG_SGD} 
	\begin{algorithmic}[1]
		\Require  $\mathcal{K}_{thres} $, stepsizes $\eta$, $r$, accuracy $\epsilon$ and epoch $T$
		\State $t_{noise} = 0$
		\For{$k=0,1,2,.., T$}
		\State $\tilde{\mu} = 1/n \sum_{i=1}^{n} \nabla f_i(\tilde{x}^k)$
		\If { $t_{noise} \geq \mathcal{K}_{thres}$ $\&\&$  {$\|\tilde{\mu}\| \leq \epsilon$ } \label{Line 9}}
		\State Randomly pick $i \in \{1,\dots,n\}$ 
		\State $\tilde{x}^k  = \tilde{x}^k - r( \nabla f_{i}(\tilde{x}^k ) ) $
		\State $t_{noise} = 0 $
		\State $\tilde{\mu} = 1/n \sum_{i=1}^{n} \nabla f_i(\tilde{x}^k)$
		\EndIf
		\State $x_0^k = \tilde{x}^k$
		\State Generate $N_k \sim $ Geom( $n/(n+b)$) 
		\For{$t=1,2,\dots, N_k$}
		\State Sample $I_t \subset \{1,...,n\}$,where $|I_t|=b$ 
		\State $x_t^k = x_{t-1}^k - \eta (\nabla f_{I_t}(x_{t-1}^k) - \nabla f_{I_t}( \tilde{x}^k )  + \tilde{\mu} ) $
		\EndFor
		\State set $\tilde{x}^{k+1} = x_t^k  $
		\State $t_{noise} = t_{noise} + 1 $
		\EndFor
	\end{algorithmic}
\end{algorithm}
\end{minipage}
\end{center}

As shown in Algorithm \ref{alg:SVRG_SGD}, our algorithm consists of an outer loop (Lines 2-18) and an inner loop (Lines 12-15). In each iteration of the outer loop, a full gradient is calculated (Line 3). Following SCSG, an integer is sampled from a geometric distribution and then used as the number of iterations in the inner loop (Line 11). Variance reduction steps (Lines 13 - 14) are commonly used in the family of SVRG. More importantly, Lines 4-9 are the steps for escaping saddle points where Line 4 is used to determine if the $\epsilon$-first-order stationary condition ($\|\tilde{\mu}\| \leq \epsilon$) is satisfied; if yes, and the escaping-saddle-point steps have not been invoked in the past $\mathcal{K}_{thres}$ epochs, perturbation by one step of SGD (Line 6) is taken. $t_{noise}$ is a counter counting the number of epochs after a perturbation and the threshold parameter $\mathcal{K}_{thres}$ is algorithm-specific (to be discussed in proofs). Theoretically, a parameter $ f_{thres}$ 
 specifying the desirable decrease of function value should be provided to terminate the algorithm. In other words, we run $\mathcal{K}_{thres}$ epochs after the SGD perturbation step, and then check if the objective $f$ decreases less than $f_{thres}$; if yes, terminate and the algorithm arrives at an ($\epsilon$, $\sqrt{\rho}\epsilon^{\frac{2}{5}}$)-second-order stationary point. In practice, there is no need to specify the value of $f_{thres}$, and we can terminate the algorithm when $f$ does not decrease for a couple of epochs.

\noindent\textbf{Convergence analysis.}
We first prove in Theorem \ref{mainTheroem} that  the CNC-SCSG (Algorithm \ref{alg:SVRG_SGD}) converges to a second-order stationary point of Problem (\ref{finite_sum}) within a finite number of epochs with high probability.
We then derive the convergence rate and  computational complexity in Theorem \ref{cor2}.
\begin{theorem} \label{mainTheroem}
	Under Assumptions \ref{assumption4f} and \ref{CNC_assumption}, Algorithm \ref{alg:SVRG_SGD} obtains an ($\epsilon$, $\sqrt{\rho}\epsilon^{\frac{2}{5}}$)-second-order stationary point within 
	$ T = \frac{b}{n}\frac{288 C C_1 l^4 ( f( x_0) - f^*)}{C_2\delta \eta^2 \tau^2 \epsilon^2} \log( \frac{1}{ \sqrt{\rho}\epsilon^{\frac{2}{5}}}) $ epochs with probability at least $1-\delta$, where $\eta L = \gamma (\frac{b}{n})^{\frac{2}{3}}$, $b\leq \frac{n}{8} $ ,  $C = b{ [\frac{( b - \frac{\eta^2 L^2 n}{b} - \eta n )( 1 - \eta L)}{ 1 + 2 \eta} - \frac{L^3\eta^2 n}{2b}] }^{-1}$,  $\gamma$  and $C_1$ are constants independent of $\epsilon$, and $C_2 = min( \frac{1 }{2},  \frac{\eta }{C L})$. 
\end{theorem}

The proof of Theorem \ref{mainTheroem} is sketched in Section \ref{proofsketch} and its complete proof is provided in Appendix B.
\begin{theorem}\label{cor2}
Under the same assumptions of Theorem \ref{mainTheroem}, the convergence rate is $\tilde{O}( \epsilon^{-2} log( 1/\epsilon))$. If we further require $b = \frac{n}{C_4}$, where $C_4 \geq 8$ is a constant and is independent of $n$, the expectation of overall computational complexity (i.e., the number of IFO calls) is 
		 $O( n\epsilon^{-2} log( 1/\epsilon))$.
\end{theorem}
Provided that the constants $C$, $C_1$, $C_2$ and $\gamma$ in Theorem \ref{mainTheroem} are independent of $\epsilon$, the number of iterations to reach convergence is $E[\sum_{k=1}^{T} N_k]=\frac{n}{b}T =\tilde{O}( \epsilon^{-2} log( 1/\epsilon))$. If we further require $C_4 = \frac{n}{b}$ to be a constant independent of $n$, the number of epochs $T$ only depends on $\frac{1}{C_4}=\frac{b}{n}$. Therefore, the number of IFO calls can be derived by $E[\sum_{k=1}^{T} (n + b N_k )]=2nT=O( n\epsilon^{-2} log( 1/\epsilon))$.

\section{A general framework for CNC-SCSG}
In this section, we reorganize the CNC-SCSG algorithm (Algorithm \ref{alg:SVRG_SGD}) into a general framework - Algorithm \ref{alg:SVRG_escape-A}, similar to   Neon \citep{xu2017first} or Neon2  \citep{allen2017neon2}, which allows to employ any first order stochastic algorithm $\mathcal{A}$ to speed up the process of finding second-order stationary points. We interpret algorithm $\mathcal{A}$  as an ``epoch" based algorithm. For example, algorithm $\mathcal{A}$ can be one epoch of SVRG, $n$ steps of SGD or one 

\begin{center}
	\begin{minipage}[t]{.7\textwidth}
		\begin{algorithm}[H]
			\caption {CNC-SCSG-Escaping}
			\label{alg:SVRG_escape_module} 
			\begin{algorithmic}[1]
				\Require $x$, $\mathcal{K}_{thres} $, stepsize $\eta$ and $r$
				\State Randomly pick $i \in \{1,\dots,n\}$ and update weight 
				\State $\tilde{x}^0  = x  - r( \nabla f_{i}(x ) ) $
				\For{$k=0,1,2,..\mathcal{K}_{thres}$}
				\State $\tilde{\mu} = 1/n \sum_{i=1}^{n} \nabla f_i(\tilde{x}^k)$
				\State $x_0^k = \tilde{x}^k$
				\State Generate $N_k \sim $ Geom( $n/(n+b)$) 
				\For{$t=1,2,\dots, N_k$}
				\State Randomly pick $I_t \subset \{1,...,n\}$
				\State $x_t^k = x_{t-1}^k - \eta( \nabla f_{I_t}(x_{t-1}^k) - \nabla f_{I_t}( \tilde{x}^k ) + \tilde{\mu} )$
				\EndFor
				\State set $\tilde{x}^{k+1} = x_t^k  $
				\EndFor
				\State $y = \tilde{x}^{\mathcal{K}_{thres}+1}$
			\end{algorithmic}
		\end{algorithm}
	\end{minipage}
\end{center}

\noindent step of GD. If we run 
$\mathcal{A}$ of SGD for instance, we do not mean the entire SGD process, but $n$ steps of SGD, and hence reasonably assume that the number of IFO calls $T_{\mathcal{A}}$ for algorithm $\mathcal{A}$ is upper bounded by $O(n)$.  Note that Algorithm \ref{alg:SVRG_SGD} consists of two parts: 1) the SGD perturbation step (Lines 4-9) to inject noise and the $\mathcal{K}_{thres}$ SCSG epochs following 
this perturbation step; 2) the remaining SCSG epochs. The first part can be taken out to form Algorithm \ref{alg:SVRG_escape_module}, and will be triggered whenever an $\epsilon_g$-first-order stationary point is satisfied in the main Algorithm \ref{alg:SVRG_escape-A}. The whole algorithm terminates when the first-order condition is satisifed and $f$ decreases less than $f_{thres}$.

In Algorithm \ref{alg:SVRG_escape-A}, the first-order algorithm $\mathcal{A}$ is applied to decrease the function value in each epoch $k$, until the norm of the gradient becomes small. Once a critical point $y_k$ s.t. $\|\nabla f(y_k) \| \leq \epsilon_g $ is reached, Algorithm \ref{alg:SVRG_escape_module} is invoked, and returns a new iterate $x_k$, which guarantees a sufficient decrease in the value of $f$, unless $y_k$ is already
an $(\epsilon_g, \epsilon_h)$-second-order stationary point. This general framework, similar to Neon and Neon2, also requires the evaluation of the first-order condition (Line 4). For stochastic algorithms that do not compute full gradients, a few approximation schemes have been developed to effectively check the first-order condition \citep{xu2017first,allen2017neon2} (see Appendix C).

\begin{center}
\vspace{-0.8cm}
\begin{minipage}[t]{.7\textwidth}
\begin{algorithm}[H]
	\caption {$\mathcal{A}$ with CNC-SCSG }
	\label{alg:SVRG_escape-A} 
	\begin{algorithmic}[1]
		\Require  $ \epsilon_g$, $ \epsilon_h$, stepsize $\eta$ and $r$
		\State  $\mathcal{K}_{thres}\leftarrow\tilde{O}(\epsilon_h^{-1}log(\epsilon_h^{-1}))$
		$f_{thres}\leftarrow\tilde{O}(\epsilon_h^4)$
		\For{$k=0,1,2,...$}
		\State $y_k = \mathcal{A}( x_k)$
		\If{ $\epsilon_g$-first-order condition }
		\State $x_k$ = CNC-SCSG-Escaping($y_k$, $\mathcal{K}_{thres}$, $\eta$, $r$)
		\If{ $f(y_k) - f(x_k) \leq f_{thres} $}
		\State return $y_k$
		\EndIf
		\Else
		\State $x_k = y_k$
		\EndIf
		\EndFor
	\end{algorithmic}
\end{algorithm}
\end{minipage}
\end{center}


\begin{remark}The framework facilitates searching for a better dimension-free 
convergence rate, and enjoys the same benefit as in Neon \citep{xu2017first} and Neon2  \citep{allen2017neon2}. But they contain parameters, e.g., $t$ and $\mathcal{F}$ in the Neon, and $T$ and a distance $r$ in the Neon2, which so far there is no knowledge on how to set up in practice. For the CNC-SCSG,  only $\mathcal{K}_{thres}$ is required to set up and $f_{thres}$ can be omitted in practice. We will provide a guidance in Section 6. 
\end{remark}

\noindent\textbf{Convergence analysis.}
We first guarantee that there is a sufficient decrease in the function value whenever the CNC-SCSG-Escaping procedure (Algorithm \ref{alg:SVRG_escape_module}) is invoked. 
\begin{lemma} \label{theroem:escaping}
	Under Assumptions \ref{assumption4f} and \ref{CNC_assumption}, and $\lambda_{min} ( \nabla^2 f( x )) \leq \epsilon_h$ , in $\mathcal{K}_{thres} = \hat{O}( \epsilon_h^{-1}log( \epsilon_h^{-1}))$ epochs, Algorithm \ref{alg:SVRG_escape_module} returns $y$, where
	$f( x ) - f( y ) \geq \frac{\eta \tau^2 \epsilon_h^4}{144l^4 \rho^2 C} min( \frac{1}{2}, \frac{\eta}{CL})$. The expected  decrease of function value per epoch is lower bounded by
	$D_e = \tilde{O}( \epsilon_h^{5} log^{-1}( \epsilon_h^{-1})).$
\end{lemma}
Lemma \ref{theroem:escaping} guarantees in expectation that the decrease of function value per epoch around strict saddle points is at least $D_e$. Then if we assume that i) when $\|\nabla f( x)\| > \epsilon_g$, in expectation $ \mathcal{A} $ drops function value by $D_{\mathcal{A}( \epsilon_g)}$ per epoch, and ii) when  $\|\nabla f( x) \|\leq \epsilon_g$ and $\lambda_{min}( \nabla^2 f( x) ) \leq  \epsilon_h$, $ \mathcal{A} $ is stable, i.e., does not increase function value too much in expectation, we can obtain the following convergence result by setting $g_{thres} = min(D_e, D_{\mathcal{A}( \epsilon_g)})$. Its proof is similar to that of Theorem \ref{mainTheroem} and has been moved to Appendix. 
\begin{theorem} \label{theroem:generic}
	Given Assumptions \ref{assumption4f} and \ref{CNC_assumption}, a first order algorithm $ \mathcal{A}$ and $y = \mathcal{A}(x)$ such that:
	\begin{align*}
	    \|\nabla f( x ) \| \ge \epsilon_g &\Rightarrow E[ f( y) - f( x )]  \leq -D_{\mathcal{A}( \epsilon_g)}, \\
		\|\nabla f( x ) \| \leq \epsilon_g &\Rightarrow E[ f( y) - f( x )]  \leq \frac{min(D_e, D_{\mathcal{A}( \epsilon_g)} ) \delta }{2} ,
	\end{align*}
	 Algorithm \ref{alg:SVRG_escape-A} converges to an $(\epsilon_g, \epsilon_h )$-second-order stationary point in 
	$T = \tilde{O}( max( \frac{1}{D_e}, \frac{1}{D_{\mathcal{A}( \epsilon_g)}}))$
	epochs with probability  $1- \delta$.
\end{theorem}
\begin{theorem} \label{theroem:generic_IFO}
	Assume all the assumptions in Theorem \ref{theroem:generic} hold, the computational complexity (i.e., the number of IFO calls) is   
	$ O( max( \frac{1}{D_e}, \frac{1}{D_{\mathcal{A}( \epsilon_g)}}) \cdot (max( T_e, 
	+ T_{check}/\mathcal{K}_{thres}, 
	T_{\mathcal{A}}+ T_{check})))$, where $T_e$ is the number of IFO calls of Algorithm \ref{alg:SVRG_escape_module} per epoch, $T_{\mathcal{A}}$ is the number of IFO calls of $\mathcal{A}$ per epoch, and $T_{check}$ is the number of IFO calls of checking the first-order stationary condition.
\end{theorem}
The computational cost $T_{check}$ of checking first-order stationary condition can be considered either in one epoch of $\mathcal{A}$ or in $\mathcal{K}_{thres}$ epochs of SCSG in Algorithm 3. Therefore, the upper bound on the number of IFO calls per epoch, including those of Line 4 in Algorithm 4, is  $max( T_e + T_{check}/\mathcal{K}_{thres}, 
	T_{\mathcal{A}}+ T_{check})$.
The computational complexity follows from multiplying the above bound with the number of epochs in Theorem \ref{theroem:generic}.
Note that there exist many first-order algorithms satisfying the requirements of algorithm $\mathcal{A}$.
For completeness, we provide an exemplar algorithm in Appendix D. 

\section{Proof sketches}
\label{proofsketch}

\begin{table}[H]
	\begin{center}
		\caption{Constraints for parameters used in the proof}
		\label{table:parameter_setup}
		\begin{tabular}{|c|c|c|}
			\hline
			Notation & Constraint Requirements & Reference \\
			\hline
			$\epsilon_g$ & $\epsilon_g = \epsilon $ &\\
			$\epsilon_h$ & $\epsilon_h = ( \rho  \epsilon)^\frac{2}{5} $ & \\
			$\gamma$ &$\gamma \leq min\{\eta_0 L(\frac{n}{b})^{\frac{2}{3}},  \frac{1}{3} \}$ &  \\
			$\eta$ & $\eta = L\gamma \frac{b}{n}^{\frac{2}{3}}$ & SCSG stepsize\\
			$r$& $r  \leq min( \frac{1 }{2},  \frac{\eta }{C L}) \frac{\tau}{12 \rho l^3}\epsilon_h^2 $ &   SGD stepsize  \\
			$f_{thres}$ & $f_{thres} \leq \frac{\eta \tau r \epsilon_h^2}{12 l \rho C} $ & Lemma \ref{eigenvalue_m}\\
			$\mathcal{K}_{thres}$ & $ \mathcal{K}_{thres} \geq \frac{C_1}{\eta \epsilon_h} \frac{b}{n}\log( \frac{1}{\epsilon_h}) $&  Lemma \ref{eigenvalue_m}\\
			$g_{thres}$ & $g_{thres}\leq \frac{\gamma}{5L} ( \frac{n}{b})^\frac{1}{3} \epsilon_g^2$ & Inequality (\ref{decrease_largegradient}) \\
			$g_{thres}$ & $ g_{thres} \geq  \frac{10l^2\gamma^2}{L \delta}(\frac{b}{n})^{\frac{1}{3}}$ & Inequality (\ref{increase_function_value}) \\ 
			$g_{thres}$ & $ g_{thres} \leq \frac{n}{b}\frac{\eta^2 \epsilon_h^3 \tau r}{C_1 12 l \rho C} log^{-1}( \frac{1}{\epsilon_h})$  & Inequality (\ref{decrease_saddle_points}) \\ 
			\hline
		\end{tabular}
	\end{center}
\end{table}

In order to prove Theorem \ref{mainTheroem}, we consider three types of epochs based on the magnitude of the gradient
   and the most negative eigenvalue of Hessian matrix at each snapshot $\tilde{x}^k$ for $k=1,2,3, ...$: i) $\tilde{x}^k$ with large gradient $\|\nabla f( \tilde{x}^k)\|\geq \epsilon_g$; ii) $\tilde{x}^k$ with large negative curvatur $\lambda_{min}(\nabla^2 f( \tilde{x}^k)) \leq -\epsilon_h$; and iii) $\tilde{x}^k$ satisfying neither of them. The following analysis is epoch-based and mainly inspired by the paper \citep{daneshmand2018escaping}, but our analysis based on SCSG is more difficult. The theoretical analysis is based on a specific setup for hyper-parameters, as  
summarized in Table \ref{table:parameter_setup}. 
Proofs of the lemmas used in this section have been moved to Appendix.

\noindent\textbf{Regime with large gradients.}
Considering that $\epsilon_g$-first-order condition is not satisfied at $\tilde{x}^{k}$, we can get the expected  function value decrease per epoch at large gradient regime as in the following lemma:
\begin{lemma} \label{decrease_LargeGrad}
	With conditions: $\eta L= \gamma ( \frac{b}{n})^\frac{2}{3}$ where $b \geq 1$, $n\geq 8b $ and $\gamma \leq \frac{1}{3}$ and assumption that $\epsilon_g$-first-order condition is not satisfied 
	, the decrease of function value per  epoch satisfies that 
	$E[f(\tilde{x}^k) - f(\tilde{x}^{k-1})] \leq -\frac{\gamma}{5L} ( \frac{n}{b})^\frac{1}{3} \epsilon_g^2.$
\end{lemma}
By setting parameter $g_{thres}\leq \frac{\gamma}{5L} ( \frac{n}{b})^\frac{1}{3} \epsilon_g^2$, we further have that
\begin{equation} \label{decrease_largegradient}
E[f(\tilde{x}^k) - f(\tilde{x}^{k-1})] \leq -g_{thres},
\end{equation}
which implies the average decrease of function value for each epoch is no smaller than $g_{thres}$.

\noindent\textbf{Regime with large negative curvatures.} In this part, we establish the number of epochs ($\mathcal{K}_{thres}$) required for a significant decrease of function value ($f_{thres}$) around strict saddle points. 
\begin{lemma}\label{eigenvalue_m}
	If the Hessian matrix at $\tilde{x}^0$ has a small negative eigenvalue, i.e. 
	$ \lambda_{min}( \nabla f^2(\tilde{x}^0)) \leq -\epsilon_h$
	,  $f_{thres} \leq \frac{\eta \tau r \epsilon_h^2}{12 l \rho C} $ and $ \mathcal{K}_{thres} \geq \frac{C_1}{\eta \epsilon_h} \frac{n}{b}\log( \frac{1}{\epsilon_h}) $, where $C_1 $ is a sufficient large constant and independent of $\epsilon_h$,
	the expectation of the function value decrease is 
	$E[f(\tilde{x}^0) - f(\tilde{x}^{\mathcal{K}_{thres}})]  \geq f_{thres}.  $
	
\end{lemma}
After imposing constraints that $f_{thres} \leq \frac{\eta \tau r \epsilon_h^2}{12 l \rho C} $ and $ \mathcal{K}_{thres} \geq \frac{C_1}{\eta \epsilon_h} \frac{b}{n}\log( \frac{1}{\epsilon_h}) $, we can find constraint for average decrease of function value per epoch:
\begin{align} 
g_{thres} =  \frac{f_{thres}}{\mathcal{K}_{thres}} \leq \frac{n}{b}\frac{\eta^2 \epsilon_h^3 \tau r}{ 12 l \rho C C_1} log^{-1}( \frac{1}{\epsilon_h}). \label{decrease_saddle_points}
\end{align}
\textbf{Regime satisfying neither of above.} We prove that the function value around a second-order stationary point does not increase too much due to disturbance.
\begin{lemma}
	Given that
	$\eta L= \gamma ( \frac{b}{n})^\frac{2}{3}$, 
	where $b \geq 1$, $n\geq 8b $ and $\gamma \leq \frac{1}{3}$, for each epoch of the SCSG, the increase of function value caused by disturbance is,
	$E[f(\tilde{x}^k) - f(\tilde{x}^{k-1})]  \leq \frac{5l^2\gamma^2}{L}(\frac{b}{n})^{\frac{1}{3}}.$
\end{lemma}
By setting parameter $ g_{thres} \geq  \frac{10l^2\gamma^2}{L \delta}(\frac{b}{n})^{\frac{1}{3}}$, we can further obtain that:
\begin{equation} \label{increase_function_value}
E[f(\tilde{x}^k) - f(\tilde{x}^{k-1})] \leq \frac{\delta g_{thres}}{2}.
\end{equation}

\noindent\textbf{Proof sketch of Theorem \ref{mainTheroem}}
With appropriate values of $\gamma$ and $C_1$, we have that $g_{thres} = \frac{n}{b}\frac{\eta^2 \epsilon_h^3 \tau r}{C_1 12 l \rho C} log^{-1}( \frac{1}{\epsilon_h})$ and Inequalities (\ref{decrease_largegradient}), (\ref{decrease_saddle_points}) and (\ref{increase_function_value}) hold.
We define event 
$ A_t := \{\|\nabla f( \tilde{x}^{t} ) \| \geq \epsilon_g ~or~ \lambda_{min}( \nabla^2 f( \tilde{x}^{t}  )) \leq -\epsilon_h\}.$
Let $R$ be a random variable representing the ratio of second-order stationary points visited by our algorithm at snapshots in the past T epochs. That is,
$R = \frac{1}{T}\sum_{t =1 }^{T} \mathbf{I} (A_t^c),$
where $\mathbf{I}(\cdot)$ is an indicator function. Let $P_t$ be the probability that $A_t$ occurs and $1-P_t$ be the probability that its complement $A_t^c$ occurs. Our goal is that
$E( R) = \frac{1}{T} \sum_{t=1}^{T}( 1 - P_t)$ holds
with high probability. That is,
$E( R) = \frac{1}{T} \sum_{t=1}^{T}( 1 - P_t) \geq 1- \delta, $
or
$\frac{1}{T}\sum_{t=1}^{T} P_t \leq \delta.$

Unfortunately, estimating the probabilities $P_t$ for $t \in \{1, ..., T\}$ is very difficult. However, 
we successfully obtain an upper bound:
$\frac{1}{T}\sum_{t=1}^{T} P_t \leq \frac{f( x_0) - f^*}{Tg_{thres}} + \frac{\delta}{2}.$
Then by letting $\frac{f( x_0) - f^*}{Tg_{thres}} + \frac{\delta}{2} \leq \delta$, we have that
\begin{align*}
T = \frac{2( f( x_0) - f^*)}{\delta g_{thres}} 
= \frac{b}{n} \frac{288 C C_1 l^4 ( f( x_0) - f^*)}{C_2\delta \eta^2 \tau^2 \epsilon^2} \log( \frac{1}{ \sqrt{\rho}\epsilon^{\frac{2}{5}}}).
\end{align*}


\section{Empirical evaluations}
In this section, we present simulation results to verify that the proposed CNC-SCSG can escape saddle points faster than its competitors.
In the simulations, we consider the following nonconvex finite-sum problem:
$\min_{x \in \mathbb{R}^d } {f( x ) = \frac{1}{n}\sum_{i=1}^{n} \frac{1}{1+ exp( - y_i z_i^T x)} + \lambda\sum_{j =1 }^{d} \frac{x_j^2}{1+ x_j^2}}$
where $y_i \in \{0, 1\}$ denotes the label, $z_i \in \mathbb{R}^d$ denotes the feature data of the $i^{th}$ data point and $\lambda$ is a positive trade-off parameter.
Data points $\{z_i\}$ with labels $y_i$=$0$ were generated from a normal distribution $N( \mu_0$=$\mathbf{0}, I)$ while data points $\{z_i\}$ with labels $y_i$=$1$ were generated from $N( \mu_1$=$\mathbf{1}, I)$.
The efficiency for escaping saddle points is measured by the number of epochs each algorithm uses for the escape.

We generated two simulation datasets: a low dimensional dataset with $n= 40 $ and $d = 4$, and a high dimensional dataset with $n= 200$ and $d=20$. 
We compared CNC-SCSG with three algorithms PGD \citep{jin2017escape}, CNC-SGD and CNC-GD \citep{daneshmand2018escaping}.
We did not include Neon-based (Neon2-based) algorithms \citep{xu2017first,allen2017neon2} because Neon (Neon2) so far provided no knowledge on how to set up the parameters $t$ and $F$ ($T$ and a distance $r$, resp.) in practice, and no empirical tests of the full algorithms were given in their original papers except in Neon (simulation was given only around saddle points which did not involve $t$ and $F$). However, without $t$ and $F$ ($T$ and $r$), we cannot correctly implement the algorithm for comparison. The complete experimental setups for each of the tested algorithms can be found in Appendix A.

\begin{figure}[t]
	\centering
	\includegraphics[width=0.9\textwidth]{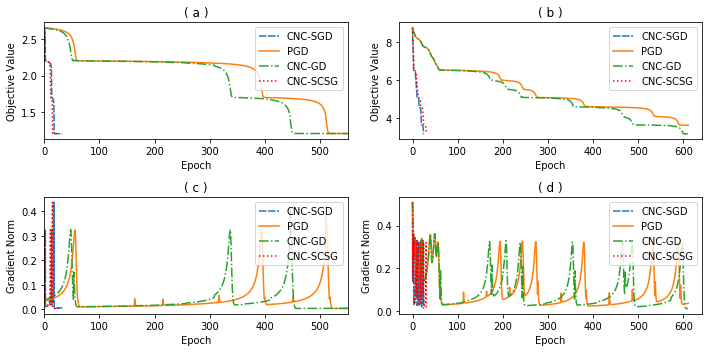}
	\caption{Comparison of different methods on a synthetic problem. Plots (a) and (b) correspond to the objective value over the increase of epochs and plots (c) and (d) correspond to the gradient norm over number of epochs. Plots (a) and (c) are for the low dimension dataset, and plots (b) and (d) are for the high dimension dataset. 
	}
	\label{fig:1}
\end{figure}

\textbf{Results.} As shown in Figure \ref{fig:1}, for the tested datasets, our proposed CNC-SCSG algorithm consistently used significantly smaller epochs to escape saddle points comparing with PGD and CNC-GD. For example in Figure \ref{fig:1}(a), for the saddle point at objective value around 2.2, CNC-GD and PGD used $\sim$250 and $\sim$350 epochs to escape it respectively while CNC-SCSG used only $\sim$10 epochs.
Compared with CNC-SGD, CNC-SCSG have comparable numbers of epochs although in theory its convergence rate is better than CNC-SGD.
We conjecture that it is because the objective function tested in the simulation is relatively simple and has a small number of intervals with large gradients.
As in  Figure \ref{fig:1}(c) and (d), we also observed that CNC-based algorithms including CNC-SCSG are more stable than isotropic noise perturbed algorithm PGD, by noticing that PGD shows spikes in both simulations.



\begin{remark}
\vspace{-3pt}
	Considering that $N_k$ can be large, wasting stochastic gradient calculation may occur. 
	Though Algorithm \ref{alg:SVRG_SGD} theoretically requires the first-order condition check (in the proof),  in practice we can remove the checking  and insert one step SGD for every $t_{thres}$ iteration, where $t_{thres}$ is moderately smaller than $n$, e.g. $n/4$,  Hence, Lines 4-9 and Line 17 in Algorithm \ref{alg:SVRG_SGD} can be moved into inner loop when first condition requirement check is removed. 
    without hurting the algorithm due to the stability of SGD. 
	( unlike PGD shows spikes in Figure \ref{fig:1} ( c ) and ( d ))
\end{remark}

\section{Conclusion}
We propose the CNC-SCSG algorithm to escape saddle points for nonconvex optimization so it converges to second-order stationary points. The CNC-SCSG algorithm can achieve the same convergence rate as CNC-GD.
We also generalize the CNC-SCSG algorithm into a framework like the Neon  or Neon2 method, which can be readily used in conjunction with any existing algorithm that converges to first-order stationary points. We observe that in all existing first-order oracle algorithms, when to break from the escaping process and the evaluation of conditions for second-order stationary points can be issues due to heavy dependence on algorithm parameters, which are often clueless to find appropriate values in practical non-convex optimization problems. We will investigate this problem as the future work.
\section*{Acknowledgments}

This work was funded by NSF grants CCF-1514357, DBI-1356655, and IIS-1718738 to Jinbo Bi, who was also supported by NIH grants K02-DA043063 and R01-DA037349.
\bibliography{nonconvex} 
\bibliographystyle{plain}
\newpage
\appendix 
\section{Comparison of algorithms for non-convex optimization}
\begin{table}[H]
	\tiny
	\renewcommand{\arraystretch}{2}
	\begin{center}
		\caption{Algorithms for Non-convex Optimization }
		\begin{tabular}{p{1.2cm}|p{1.cm}|p{3.4cm}|p{0.5cm}|p{0.9cm}|p{4.5cm}|p{1.9cm}}
			\specialrule{.2em}{.1em}{.1em} 
			Guarantees & Oracle & Algorithm & $\epsilon_g$ & $\epsilon_h$ & Time Complexity & Dependence \\
			&&&&&( iteration $\cdot$ Oracle/iter ) &\\
			\specialrule{.15em}{.07em}{.07em}
			\multirow{4}{*}{\shortstack{First-order\\ Stationary \\ Point}}  
			& \multirow{4}{*}{ Gradient } & GD;  SGD  \citep{nesterov2013introductory} &  $\epsilon$& no & $O(\epsilon^{-2})\cdot O(n)$; $O(\epsilon^{-4})\cdot O(1)$  &	N/A\\
			\cline{3-7}
			& & AGD; SAGD \citep{ghadimi2016accelerated}& $\epsilon$ &no  &	$O(\epsilon^{-2})\cdot O(n)$; $O(\epsilon^{-4})\cdot O(1)$&N/A\\ 
			\cline{3-7}
			& & Guarded-AGD \citep{carmon2017convex} &$\epsilon$ &no & $O(\epsilon^{-\frac{7}{4}}\log(\epsilon^{1}))\cdot O(n)$ &N/A\\
			\cline{3-7}
			& & SVRG \citep{reddi2016stochastic} &$\epsilon$ &no & $O(\epsilon^{-2})\cdot O(n)$  &N/A\\
			\cline{3-7}
			& & SCSG \citep{lei2017non} &$\epsilon$ &no &  $O(min(\epsilon^{-\frac{10}{3}},n^{\frac{2}{3}}\epsilon^{-2}))$ &N/A\\
			
			\specialrule{.2em}{.05em}{.05em} 
			
			\multirow{7}{*}{\shortstack{2nd-order \\ Stationary\\ Point \\ ( Local \\ Minimum )} }
			& \multirow{5}{*}{Gradient} &PSGD \citep{ge2015escaping}&  $\epsilon$&$\sqrt{\rho\epsilon}$& $O(\epsilon_g^{-4})\cdot O(1)$ & $poly(d)$\\
			\cline{3-7}
			&	& PGD  \citep{jin2017escape} & $\epsilon$ & $\sqrt{\rho\epsilon}$ & $O(\epsilon^{-2}\log (\frac{d}{\epsilon}))\cdot O(n)$ & $poly\log(d)$ \\
			\cline{3-7}
			&	& PAGD  \citep{jin2018accelerated} & $\epsilon$ &$\sqrt{\rho\epsilon}$  &$O(\epsilon^{-\frac{7}{4}}\log^6 (\frac{d}{\epsilon}))\cdot O(n)$ &$poly\log(d)$\\
			\cline{3-7}
			&	& CNC-PGD \citep{daneshmand2018escaping} & $\epsilon$ &$\sqrt{\rho}\epsilon^{\frac{2}{5}}$  & $O(\epsilon^{-2}\log (\epsilon^{-1}))\cdot O(n)$ & free\\
			\cline{3-7}
			&	& CNC-PSGD \citep{daneshmand2018escaping} & $\epsilon$ & $\sqrt{\rho}\epsilon^{\frac{2}{5}}$ & $O(\epsilon^{-4}\log^2 (\epsilon^{-1}))\cdot O(1)$& free\\
			\cline{3-7}
			&	& NEON+SVRG \citep{xu2017first} & $\epsilon$ & $\epsilon^{\frac{1}{2}}$ & $O(n^{\frac{2}{3}}\epsilon^{-2}+n\epsilon^{-\frac{3}{2}}+\epsilon^{-\frac{11}{3}})$ &$poly\log(d)$\\
			\cline{3-7}
			&	& NEON2+CDHS \citep{allen2017neon2}& $\epsilon_g$ & $\epsilon_h$  & $O(n^{\frac{2}{3}}\epsilon_g^{-2}+n\epsilon_h^{-3}+n^{\frac{3}{4}}\epsilon_h^{-\frac{7}{2}})$ & $poly\log(d)$ \\
			\cline{3-7}
			&	& Natasha2 by \citep{allen2018natasha}& $\epsilon_g$ & $\epsilon_h$   & $O(\epsilon_h^{-5} + \epsilon^{-\frac{13}{4}} + \epsilon_g^{-3}\epsilon_h^{-1})$& $poly\log(d)$ \\
			\cline{3-7}
			&	& SPIDER+NEON2 \citep{fang2018spider}& $\epsilon$& $\epsilon^\frac{1}{2}$  & $O(min(n^\frac{1}{2} \epsilon^{-2} + \epsilon^{-\frac{5}{2}} , \epsilon^{-3}))$& $poly\log(d)$\\
			\cline{3-7}
			&	& SNVRG+NEON2 \citep{zhou2018finding} &$\epsilon_g$ & $\epsilon_h$  &$O(n^{\frac{1}{2}}\epsilon_g^{-2} + n \epsilon_h^{-3} + n^{\frac{3}{4}} \epsilon_h^{-\frac{7}{2}})$ &$poly\log(d)$\\
			\cline{3-7}
			& & Stablized SVRG\citep{ge2019stabilized} &$\epsilon_g$& $\epsilon_h$ &$O(n^{\frac{2}{3}}\epsilon_g^{-2} + n\epsilon_h^{-3} + n^{\frac{2}{3}}\epsilon_h^{-4})$ &$poly\log(d)$\\
			\cline{3-7}
			& & SSRGD\citep{li2019ssrgd} &$\epsilon_g$& $\epsilon_h$ &$O(n^{\frac{1}{2}}\epsilon_g^{-2} + n^{\frac{1}{2}}\epsilon_h^{-4} + n\epsilon_h^{-3})$ &$poly\log(d)$\\
			\cline{2-7}
				
			& \multirow{2}{*}{ \shortstack{Hessian\\-vector}}  & NCD+AGD\citep{carmon2016accelerated} &  $\epsilon$& $\epsilon^\frac{1}{2}$  & $O(\epsilon^{-\frac{7}{4}} log( \epsilon^{-1}))\cdot T_{Hv}$&N/A \\
			\cline{3-7}
			&	& FastCubic by \citep{agarwal2017finding}&$\epsilon$& $\epsilon^\frac{1}{2}$  & $O(\epsilon^{-\frac{3}{2}} n + \epsilon^{-\frac{7}{4}} n^{\frac{3}{4}})\cdot T_{Hv}$&N/A \\
			\cline{3-7}
			&	& SCubic by \citep{tripuraneni2017stochastic} & $\epsilon$& $\sqrt{\rho \epsilon}$  & $O(\epsilon^{-\frac{7}{2}})\cdot T_{Hv}$&  N/A \\
			\cline{3-7}
			&	& SVRCubic by \citep{zhou2018stochastic} & $\epsilon$& $\sqrt{\rho \epsilon}$  & $O(\epsilon^{-\frac{3}{2}})\cdot T_{cubic}$&  N/A \\
			\cline{2-7}
			
			& \multirow{2}{*}{ \shortstack{Hessian}}  & Cubic Alg \citep{nesterov2006cubic} & $\epsilon$&$\sqrt{\epsilon}$  &$O(\epsilon^{-\frac{3}{2}})\cdot T_{cubic}$ & N/A \\
			\cline{3-7}
			&	& Trust Region \citep{curtis2017trust}&  $\epsilon$&$\sqrt{\epsilon}$  &$O(\epsilon^{-\frac{3}{2}})\cdot T_{H}$&  N/A\\

			\specialrule{.2em}{.05em}{.05em}

		\end{tabular}
	\end{center}
	\vspace{-5mm}
\end{table}
Here $T_{Hv}$ is the time for a Hassian vector product oracle and it is $O(d)$ as discussed by \citep{agarwal2017finding}, and $T_{cubic}$ is the time for each iteration of the cubic algorithm.
Under the strict saddle point assumption, any second-order stationary point will be a local minimizer.
Additionally, $\| \nabla f( x ) \| \leq \epsilon_g$,  $\lambda_{min} ( \nabla^2 f( x ) ) \geq -\epsilon_h$, and $N/A$ means not defined in the corresponding reference.

\section{The experimental setup}
In order to demonstrate the efficiency of anisotropic noise generated by SGD for high dimensional data, we generated two simulation datasets: I) a low dimensional dataset with $n= 40 $ and $d = 4$; and II) a high dimensional dataset with $n= 200$ and $d=20$. For both datasets, $\lambda=0.5$. The stepsizes for both SCSG and GD are 0.5, and the stepsizes for the SGD in both CNC-SCSG and CNC-GD are 2. Noise for PGD was uniformly drawn from the sphere of an Eucliden ball with radius 0.05. When the first order condition is satisfied, noise injection or a SGD jumping step was taken only if they had not been taken in the previous 50 iterations.  

\section{Detailed proofs}
\subsection{Proofs of lemmas}

\begin{lemma}\label{series}
For all $0<\beta<1$, we have bounded series:
\begin{equation}
\sum_{i=1}^{t}(1+\beta)^{t-i}\leq 2\beta^{-1}(1+\beta)^t;
\end{equation}
\begin{equation}
\sum_{i=1}^{t}i(1+\beta)^{t-i}\leq 2\beta^{-2}(1+\beta)^t.
\end{equation}
\end{lemma}
\begin{proof}
From Taylor expansion, for any $|z|<1$, we have:
\begin{align}
\sum_{k=1}^{\infty}(z)^{k}&\leq 1/(1-z);\nonumber\\
\sum_{k=1}^{\infty}k(z)^{k}&\leq z/(1-z)^2.\nonumber
\end{align}
\end{proof}

\begin{lemma} \label{mini_batch_variance}
	Let $x_j \in \mathbb{R}^d$ be an arbitrary population of $M$ vectors with 
	$$ \sum_{j=1}^{M} x_j = 0.$$
	Further let $J$ be a uniform random subset of $\{ 1,2,\dots, M \}$ with size $m$. Then
	$$E[\| \frac{1}{m} \sum_{j \in J} x_j \|^2] = \frac{M-m}{( M -1)m}\frac{1}{M} \sum_{j =1}^{M} \| x_j \|^2 \leq \frac{\mathbf{I}( m < M)}{m} \frac{1}{M} \sum_{j =1 }^{M}\| x_j\|^2,$$
	where $\mathbf{I}(\cdot )$ is an indicator function.
\end{lemma}

\begin{lemma} \label{Geom_dist_property}
	Let $N\sim Geom( \gamma)$. Then for any sequences $\{D_i\},$
	$$ E[D_N - D_{N+1}] = ( \frac{1}{\gamma} - 1)( D_0 - E[D_N]).$$
\end{lemma}

Both Lemmas \ref{mini_batch_variance} and \ref{Geom_dist_property} can reference \citep{lei2017non} for proof.

\subsection{One-epoch analysis for SCSG }
The analysis for this section is inspired by  \citep{lei2017non} and is included here for completeness.
\begin{lemma} \label{variance_direction}
	Denote $v_t^k = \nabla f_{I_t}(x_{t}^k) - \nabla f_{I_t}( \tilde{x} ) + \tilde{\mu}$ to be the updating direction at $t-th$ iteration of $k-th$ epoch, then
	\begin{equation}
	E_{I_t} [\| v_t^k \|^2] \leq \frac{L^2}{b} \| x_{t}^k - x_0^{k} \|^2 + \|\nabla f( x_t^k) \|^2. 
	\end{equation}
\end{lemma}

\begin{proof}
	Let $\xi_{t}^k = \nabla f_{I_t}(x_{t}^k) - \nabla f( x_{t}^k) - (\nabla f_{I_t}( \tilde{x} ) - \tilde{\mu}),$  easily derive $E_{I_t} [ \xi_t^k ] = 0$.
	
	\begin{align*}
	E_{I_t}[ \| v_t^k \|^2] &= E_{I_t}[ \| \nabla f_{I_t}( x_t^k) -  \nabla f_{I_t}( x_0^k) - ( \nabla f( x_t^k) - \nabla f( x_0^k)) + \nabla f( x_t^k) \|^2 ]\\
	&= E_{I_t}[ \| \xi_t^k + \nabla f( x_t^k) \|^2 ]\\
	&= E_{I_t} [\| \xi_t^k \|^2]+ 2<E_{I_t} [ \xi_t^k ] , \nabla f( x_t^k)>+ \|\nabla f( x_t^k) \|^2\\
	&=E_{I_t} [\| \xi_t^k \|^2]+ \|\nabla f( x_t^k) \|^2 \\
	\end{align*}
	By Lemma \ref{mini_batch_variance}:
	\begin{align*}
	E_{I_t} [\| \xi_t^k \|^2] &= E_{I_t} \| \nabla f_{I_t}( x_t^k) -  \nabla f_{I_t}( x_0^k) - ( \nabla f( x_t^k) - \nabla f( x_0^k)) \|^2 \\
	&\leq \frac{1}{b} \frac{1}{n} \sum_{i=1}^{n} \| \nabla f_{z}( x_t^k) -  \nabla f_{z}( x_0^k) - ( \nabla f( x_t^k) - \nabla f( x_0^k)) \|^2\\
	&=  \frac{1}{bn} (\sum_{z=1}^{n} \| \nabla f_{z}( x_t^k) -  \nabla f_{z}( x_0^k) \|^2 - \| ( \nabla f( x_t^k) - \nabla f( x_0^k)) \|^2) \\
	&\leq  \frac{1}{bn} \sum_{z=1}^{n} \| \nabla f_{z}( x_t^k) -  \nabla f_{z}( x_0^k) \|^2\\
	&\leq \frac{L^2}{b} \| x_{t}^k - x_0^{k} \|^2,
	\end{align*}
	where the first inequality is due to Lemma \ref{mini_batch_variance} and last inequality is due to L-smoothness.
\end{proof}

\begin{lemma}\label{inequality_1}
	Suppose $\eta L < 1$, then we can get
	\begin{equation} \label{inquality:1}
	\eta n ( 1 - \eta L ) E[\| \nabla  f( \tilde{x}^k)  \|^2] \leq b E[ f( \tilde{x}^{k-1}) - f( \tilde{x}^k)] + \frac{L^3\eta^2n}{2b} E[\| \tilde{x}^k - \tilde{x}^{k-1} \|^2].
	\end{equation}
\end{lemma}

\begin{proof}
	\begin{align*}
	E_{I_t} [f( x_{t+1}^k)] &\leq f(x_t^k ) - \eta \| \nabla f( x_t^k)\|^2 + \frac{L \eta^2}{2} E_{I_t} [\| v_t^k \| ^2] \\
	&\leq f( x_t^k) - \eta( 1 - \frac{\eta L}{2}) \| \nabla f( x_t^k) \|^2 + \frac{L^3\eta^2}{2b}\| x_t^k - x_0^k\|^2\\
	&\leq f( x_t^k) - \eta( 1 - \eta L) \| \nabla f( x_t^k) \|^2 + \frac{L^3\eta^2}{2b}\| x_t^k - x_0^k\|^2,
	\end{align*}
	where the first inequality is due to L-smoothness and the second inequality is due to Lemma \ref{variance_direction}.
	
	Let $E_k$ denotes the expectation over all  mini-batchs $I_0, I_1, ...$ of epoch $k$ given $N_k$. Hence, 
	$$\eta ( 1 - \eta L) E_k[\| \nabla f( x_t^k) \|^2] \leq E_k[ f( x_t^k)] - E_k[ f(x_{t+1}^k)] + \frac{L^3\eta^2}{2b}E_k[\| x_t^k - x_0^k\|^2].$$
	Let $t = N_k$. By taking expectation with respect to $N_k$ and using Fubini's theorem, we obtain that
	\begin{align*}
	\eta ( 1 - \eta L)E_{N_k} E_k[\| \nabla f( x_{N_k}^k) \|^2] &\leq E_{N_k}( E_k[ f( x_{N_k}^k)] -  E_k[ f(x_{{N_k}+1}^k)]) + \frac{L^3\eta^2}{2b}E_{N_k}E_k[\| x_{N_k}^k - x_0^k\|^2] \\
	&= \frac{b}{n}( f( x_0^k) - E_k E_{N_k}[f( x_{N_K}^k)]) + \frac{L^3 \eta^2}{2b}E_kE_{N_k}[\| x_{N_k}^k - x_0^k\|^2],
	\end{align*}
	where the equality is due to Lemma \ref{Geom_dist_property}.\\
	Replace $x_{N_k}^k , x_0^k$ with $\tilde{x}_{k+1}, \tilde{x}_k$, we get the desired inequality.
\end{proof}

\begin{lemma}\label{inequality_2}
	Suppose $\eta^2 L^2 < \frac{b^2}{n}$, then 
	\begin{equation} \label{inquality:2}
	( b - \frac{\eta^2 L^2 n}{b}) E [\| \tilde{x}^k - \tilde{x}^{k -1}\|^2] \leq -2 \eta n E [\langle\nabla f(\tilde{x}^k), \tilde{x}^k - \tilde{x}^{k -1 } \rangle] + 2 \eta^2 n  E[\| \nabla  f( \tilde{x}^k)  \|^2].
	\end{equation}
\end{lemma}

\begin{proof}
	\begin{align*}
	E_{I_t}[\| x_{t+1}^k - x_0^k \|^2]&E_{I_t}[\| x_{t}^k -\eta v_t^k - x_0^k \|^2]\\
	&= \| x_{t}^k - x_0^k \|^2 - 2\eta \langle E_{I_t} [v_t^k], x_t^k - x_0^k \rangle + \eta^2 E_{I_t} [\| v_t^k \|^2] \\
	&= \| x_{t}^k - x_0^k \|^2 - 2\eta \langle \nabla f( x_t^k), x_t^k - x_0^k \rangle + \eta^2 E_{I_t} [\| v_t^k \|^2] \\
	& \leq ( 1 +\frac{\eta^2 L^2}{b} ) \| x_{t}^k - x_0^k \|^2 - 2\eta \langle \nabla f( x_t^k), x_t^k - x_0^k \rangle + \eta^2 \| \nabla f( x_t^k) \|^2\\
	& \leq ( 1 +\frac{\eta^2 L^2}{b} ) \| x_{t}^k - x_0^k \|^2 - 2\eta \langle \nabla f( x_t^k), x_t^k - x_0^k \rangle + 2\eta^2 \| \nabla f( x_t^k) \|^2.
	\end{align*}
	
	where first inequality dues to Lemma \ref{variance_direction} and second inequality is  for future convenience. 
	
	Using the same notation $E_k$ as in previous lemma, we get:
	$$2\eta E_k [\langle \nabla f( x_t^k), x_t^k - x_0^k \rangle] \leq ( 1 +\frac{\eta^2 L^2}{b} ) E_k [\| x_{t}^k - x_0^k \|^2] - E_k [\| x_{t+1}^k - x_0^k \|^2]  + 2\eta^2 E_k [\| \nabla f( x_t^k) \|^2].  $$
	
	Then Let $t = N_k $. By taking expectation with respect to $N_k$ and using Fubini's Theorem, we get 
	\begin{align*}
	&2\eta E_{N_k }E_k [\langle \nabla f( x_{N_k }^k), x_{N_k }^k - x_0^k \rangle]  \\
	& \leq ( 1 +\frac{\eta^2 L^2}{b} ) E_{N_k }E_k [\| x_{N_k }^k - x_0^k \|^2] - E_{N_k }E_k [\| x_{N_k +1 }^k - x_0^k \|^2]  + 2\eta^2 E_{N_k }  [\| \nabla f( x_{N_k }^k) \|^2] \\
	&=( - \frac{b}{n} + \frac{\eta^2 L^2}{b}) E_{N_k }E_k [\| x_{N_k }^k - x_0^k \|^2] + 2\eta^2 E_{N_k } [ \| \nabla f( x_{N_k }^k) \|^2].
	\end{align*}
	Replace $x_{N_k}^k , x_0^k$ with $\tilde{x}^{k+1}, \tilde{x}^k$, we get the desired result.
\end{proof}

\subsection{Lemmas for epochs with large gradient}
\begin{lemma} \label{LargeGrad}
	With conditions:
	$\eta L= \gamma ( \frac{b}{n})^\frac{2}{3}$ 
	where $b \geq 1$, $n\geq 8b $ and $\gamma \leq \frac{1}{3}$,
	\citep{lei2017non} guarantee the decrease of function value for each epoch: 
	\begin{align*} 
	E[f(\tilde{x}^k) - f(\tilde{x}^{k-1})] \leq -\frac{\gamma}{5L} ( \frac{n}{b})^\frac{1}{3} E[\| \nabla f( \tilde{x}^k) \|^2].
	\end{align*}
\end{lemma}
\begin{proof}
	Lemma \ref{inequality_1} Inequality (\ref{inquality:1}) $\times 2$ + Lemma \ref{inequality_2} Inequality (\ref{inquality:2})  $\times\frac{b}{n\eta}:$
	\begin{align*}
	&2\eta n ( 1 - \eta L - \frac{b}{n}) E[\|\nabla f(\tilde{x}^k) \|^2] + \frac{b^3 - \eta^2L^2bn - \eta^3L^3n^2}{b\eta n} E[\|\tilde{x}^k -\tilde{x}_{k-1}\|^2 ]\\
	& \leq -2b E[\langle \nabla f( \tilde{x}^k), \tilde{x}^k - \tilde{x}_{k-1}\rangle] + 2 b E[ f(\tilde{x}^{k-1} ) - f(\tilde{x}^k )] \\
	& \leq \frac{b\eta n}{b^3 - \eta^2L^2bn - \eta^3L^3n^2} b^2 E[\|\nabla f(\tilde{x}^k) \|^2] + \frac{b^3 - \eta^2L^2bn - \eta^3L^3n^2}{b\eta n} E[\|\tilde{x}^k -\tilde{x}_{k-1}\|^2] \\&  + 2 b E[ f(\tilde{x}^{k-1} ) - f(\tilde{x}^k )]. \\
	\end{align*}
	Therefore:
	$$ \frac{\eta n}{b}(  2 - \frac{2b}{n} - 2\eta L - \frac{b^3}{b^3 - \eta^2L^2bn - \eta^3L^3n^2})E[\|\nabla f(\tilde{x}^k) \|^2] \leq 2 E[f(\tilde{x}^{k-1} ) - f(\tilde{x}^k )].$$
	Since $\eta L = \gamma ( \frac{b}{n})^{\frac{2}{3}}$ and $n \geq 8b\geq 8$:
	$$b^3 - \eta^2L^2bn - \eta^3L^3n^2 \geq b^3( 1 - \frac{\gamma^2}{2} - \gamma^3).$$
	Then we get:
	$$\gamma( \frac{n}{b})^{\frac{1}{3}}( 2 - \frac{2b }{n} - 2\gamma(\frac{b}{n})^{\frac{2}{3}} - \frac{1}{ 1 - \frac{\gamma^2}{2} - \gamma^3}) E[\|\nabla f(\tilde{x}^k) \|^2] \leq 2L E[ f(\tilde{x}^{k-1} ) - f(\tilde{x}^k )].$$
	Considering $ n\geq 8b \geq 8$ and $\gamma \leq \frac{1}{3}$, we get the final result.	
\end{proof}

\subsection{Lemmas for epochs with small gradient but not around saddle points}
\begin{lemma}
	With conditions:
	$\eta L= \gamma ( \frac{b}{n})^\frac{2}{3}$ 
	where $b \geq 1$, $n\geq 8b $ and $\gamma \leq \frac{1}{3}$, for each iteration of SCSG, the increase of function value caused by disturbance is:
	$$E[f(x_t) - f( x_{t-1})]  \leq \frac{5l^2\gamma^2}{L}(\frac{b}{n})^{\frac{4}{3}}.$$
	
\end{lemma}
	
\begin{proof}
		
		\begin{align*}
		E[ \|\xi_t\|^2] &= E[ \|\xi_t^k\|^2]\\
		&= E[ \| \nabla f(x_{t-1}^k) - \nabla f_{I_t}(x_{t-1}^k) + \nabla f_{I_t}( \tilde{x}^k ) - \tilde{\mu} \|^2 ]\\
		&\leq 2E[\| \nabla f(x_{t-1}^k) - \nabla f_{I_t}(x_{t-1}^k)\|^2 + \|\nabla f_{I_t}( \tilde{x}^k ) - \tilde{\mu} \|^2 ]\\
		&\leq 2E[\| \nabla f_{I_t}(x_{t-1}^k)\|^2 + \|\nabla f_{I_t}( \tilde{x}^k ) \|^2 ]\\
		&\leq 4l^2,
		\end{align*}
where last inequality is due to the norm of gradient is $l$-bounded.\\
		\begin{align*} 
		E[f(x_t) - f( x_{t-1})]&\leq -\eta \| \nabla f( x_{t-1})\|^2 + \frac{L}{2}\eta^2 E[\|\nabla f( x_{t-1}) + \xi_{t-1}\|^2]\\
		&\leq \frac{L}{2}\eta^2 E[\|\nabla f( x_{t-1}) + \xi_{t-1}\|^2]\\
		&\leq L\eta^2 E[\|\nabla f( x_{t-1})\|^2] + L\eta^2E[\|\xi_{t-1}\|^2]\\
		& \leq 5Ll^2\eta^2 \\
		&= \frac{5l^2\gamma^2}{L}(\frac{b}{n})^{\frac{4}{3}}.
		\end{align*}	
\end{proof}

\begin{lemma}
	With conditions:
	$\eta L= \gamma ( \frac{b}{n})^\frac{2}{3}$ 
	where $b \geq 1$, $n\geq 8b $ and $\gamma \leq \frac{1}{3}$, for each epoch of SCSG, the increase of function value caused by disturbance is:
	$$E[f(\tilde{x}^k) - f(\tilde{x}^{k-1})]  \leq \frac{5l^2\gamma^2}{L}(\frac{b}{n})^{\frac{1}{3}}.$$
	
\end{lemma}

\begin{proof}
	
	\begin{align*} 
	E[f(\tilde{x}^k) - f(\tilde{x}^{k-1})] &= \sum_{t=1}^{N_k}E[f(x_t^{k-1}) - f( x_{t-1}^{k-1})]\\
	&\leq  \frac{5l^2\gamma^2}{L}(\frac{b}{n})^{\frac{4}{3}} E[N_k]\\
	&=   \frac{5l^2\gamma^2}{L}(\frac{b}{n})^{\frac{1}{3}}
	\end{align*}	
\end{proof}

\subsection{Lemmas for epochs around saddle points}
	The following representations are used to show different status of parameters in the following proof:\\
$x_t^k $ : parameter $ x$ at $t-th$ iteration of  $k-th$ epoch; \\
$\tilde{x}^k $: parameter $x$ used to calculate full gradient at epoch $k$; \\
$x_t $ : parameter $x$ at $t-th$ iteration, where $t$ is the global counter;\\
\begin{lemma} \label{decrease_saddle}
	Suppose $b \geq 1$, $n\geq 8b \geq 8$, then there exists a small stepsize $\eta$, such that 
	$$E[\| \tilde{x}^k - \tilde{x}^{k-1} \|^2] \leq C E[ f( \tilde{x}^{k-1}) - f( \tilde{x}^k)],$$
	where $ C = b{ [\frac{( b - \frac{\eta^2 L^2 n}{b} - \eta n )( 1 - \eta L)}{ 1 + 2 \eta} - \frac{L^3\eta^2 n}{2b}] }^{-1} >0. $
\end{lemma}
\begin{proof}
	With Inequality (\ref{inquality:2}) in Lemma \ref{inequality_2} and $-2\langle a, b \rangle \leq \|a\|^2 + \|b\|^2$, we have:
	\begin{align*}
	    	(b - \frac{\eta^2 L^2 n}{b}) E[\|\tilde{x}^k - \tilde{x}^{k-1}\|^2] &\leq -2\eta n E[\langle\nabla f(\tilde{x}^k ), \tilde{x}^k - \tilde{x}^{k-1}\rangle] + 2\eta^2 n E[\| \nabla f( \tilde{x}^k)\|^2]\\
	& \leq \eta n E[\|\tilde{x}^k - \tilde{x}^{k -1}\|^2]+(\eta n +2\eta^2 n) E[\| \nabla f( \tilde{x}^k)\|^2]
	\end{align*}
	Reorganize:
	$$(b-\frac{\eta^2 L^2 n}{b}-\eta n)E[\|\tilde{x}^k - \tilde{x}^{k -1}\|^2] \leq (\eta n +2\eta^2 n) E[\| \nabla f( \tilde{x}^k)\|^2].$$
	With Inequality (\ref{inquality:1}) in Lemma \ref{inequality_1}, we plug in the bound of $E[\| \nabla f( \tilde{x}^k)\|^2]$ and reorganize:
	$$\frac{(b-\frac{\eta^2 L^2 n}{b}-\eta n)(\eta n(1-\eta L))}{\eta n +2\eta^2 n}E[\|\tilde{x}^k - \tilde{x}^{k -1}\|^2] \leq bE[ f( \tilde{x}^{k-1}) - f( \tilde{x}^k)] +\frac{L^3\eta^2 n}{2b}E[\|\tilde{x}^k - \tilde{x}^{k -1}\|^2]. $$
	Hence $$E[\|\tilde{x}^k - \tilde{x}^{k -1}\|^2] \leq CE[f(\tilde{x}^{k-1})-f(\tilde{x}^k)],$$ where $C=b{[\frac{(b-\frac{\eta^2 L^2 n}{b}-\eta n)(1-\eta L)}{1 +2\eta}-\frac{L^3\eta^2 n}{2b}]}^{-1}.$

	Then exist constant $\eta_0$, for arbitrary $\eta \leq \eta_0$, the following inequality is true: 
	$$ \frac{( b - \frac{\eta^2 L^2 n}{b} - \eta n )(  1 - \eta L)}{ 1+ 2 \eta} - \frac{L^3\eta^2 n}{2b} > 0.$$
	Therefore, there exist positive constant $C$, such that 
	$$E[\| \tilde{x}^k - \tilde{x}^{k-1} \|^2] \leq C E[ f( \tilde{x}^{k-1}) - f( \tilde{x}^k)].$$
	Let's argue the existence of $\eta_0$:
	\begin{align*}
	h( \eta) & = \frac{( b - \frac{\eta^2 L^2 n}{b} - \eta n )(  1 - \eta L)}{ 1+ 2 \eta} - \frac{L^3\eta^2 n}{2b} \\
	&= \frac{b + \Theta ( \eta) + O( \eta)}{ 1 + 2 \eta}.\\
	\end{align*}
	Hence, 
	\begin{equation*}
	lim_{\eta \rightarrow +0} h( \eta)= b. 
	\end{equation*} 
	Given that $h( \eta)$ is continous in $(0, 1]$, there must exists $\eta_0$, such that for arbitrary $\eta \leq \eta_0$ , $$f( \eta) > 0.$$
\end{proof}
We want to see significant decrease in function values $f_{thres}$ in every $\mathcal{K}_{thres}$ epochs, not only when in the scenario where gradients is large, but also when the smallest eigenvalue of hessian matrix is strictly small. The proof is finished by contradiction: first assume the decrease of function value is lower-bounded, then derive an upper bound and a lower bound, then contradiction is introduced by let lower bound larger than upper bound. Hence, the decrease of function value is larger than $f_{thres}$.\\

Without loss generality, $\tilde{x}^0$ satisfies first-order stationary condition, which is the starting point of escaping saddle points stage and is followed by one step SGD and $\mathcal{K}_{thres}$ epoch of SCSG.

	\begin{lemma}\label{distance_Upper_bound}
		Suppose that expectation of the decrease function value is lower-bounded:
		\begin{align}\label{ineq:value_decrease}
		    	E[ f( \tilde{x}^{\mathcal{K}})] - f( \tilde{x}^0) \geq -f_{thres}. 
		\end{align}
		Then, the expectation of the distance from the current  $\tilde{x}^{\mathcal{K}}$ to $\tilde{x}^0$ is bounded as
		\begin{equation} \label{upperbound4distiance}
		E[\|  \tilde{x}^{\mathcal{K}} - \tilde{x}^0\|^2] \leq2\mathcal{K}C f_{thres} + \mathcal{K}C L(lr)^2 + 2( lr)^2.
		\end{equation}
	\end{lemma}

\begin{proof}
	
	$ E[f(\tilde{x}^\mathcal{K}) - f( \tilde{x}^0)]$ can be split into two parts: the first step for SGD : $\tilde{x}^{(1)} = \tilde{x}^0 - r \nabla f(  \tilde{x}^0) $and the remaining $\mathcal{K}-1$ epochs SCSG,$\tilde{x}^{(1)}, ..., \tilde{x}^{(\mathcal{K})}$ .\\
	
	For inserting SGD step:
	\begin{align} \label{ineq:sgd}
	E[f( \tilde{x}^1) - f( \tilde{x}^0)]&\leq -r \| \nabla f( \tilde{x}^0) \|^2 + \frac{L(lr)^2}{2} \leq \frac{L(lr)^2}{2}.
	\end{align}
	
   By Lemma \ref{decrease_saddle}, we can bound the expected distance in the parameter space for remaining $\mathcal{K}-1$ epochs of SCSG:
	\begin{align*} 
	E[\| \tilde{x}^{\mathcal{K}} - \tilde{x}^1 \|^2 ] & \leq (\mathcal{K}-1)\sum_{k=1}^{ \mathcal{K}-1}E[\| \tilde{x}^{k+1} - \tilde{x}^{k} \|^2 ]\\
	& \leq \mathcal{K} \sum_{k=1}^{ \mathcal{K}-1}CE[( f(  \tilde{x}^{k} ) - f( \tilde{x}^{k+1}) )]\\
	& \leq \mathcal{K} CE[( f(  \tilde{x}^{1} ) - f( \tilde{x}^{\mathcal{K}}) )]\\
	& \leq \mathcal{K} CE[( f(  \tilde{x}^{0} ) - f( \tilde{x}^{\mathcal{K}})  + f( \tilde{x}^{1}) - f(\tilde{x}^{0})) ]\\
	&\leq \mathcal{K}C f_{thres} + \frac{\mathcal{K}C L(lr)^2}{2}
	\end{align*}
	The last inequality is due to assumption \ref{ineq:value_decrease} and inequality \ref{ineq:sgd}.
	
	Replace the above inequality into the following bound:
	\begin{align*} 
	E[\| \tilde{x}^{\mathcal{K}} - \tilde{x}^{0} \|^2] &= 2(  E[\| \tilde{x}^{\mathcal{K}} - \tilde{x}^{1} \|^2 ] + E[\| \tilde{x}^{1} -\tilde{x}^{0} \|^2 ]\\
	& \leq 2(   \mathcal{K}C f_{thres} +\frac{\mathcal{K}C L(lr)^2}{2}+  ( lr)^2)\\
	& \leq  2\mathcal{K}C f_{thres} + \mathcal{K}C L(lr)^2 + 2( lr)^2.\\
	\end{align*}
	
	Hence, we obtain upper bound for $E[\| \tilde{x}^{\mathcal{K}} - \tilde{x}^0 \|^2]$.
\end{proof}

\begin{lemma}\label{distance_Upper_bound_t+1}
	[Distance Upper Bound]
	Suppose that expectation of the decrease function value is lower-bounded:
	$$E[ f( \tilde{x}^{\mathcal{K}})] - f( \tilde{x}^0) \geq -f_{thres}. $$
	Then, the expectation of the distance from the current iterate $x_{t+1}$ to $\tilde{x}^0$ is bounded as
	\begin{equation} \label{upperbound4distiance_t+1}
	E[\|  x_{t+1} - \tilde{x}^0\|^2] \leq 4\mathcal{K}C f_{thres} + 2\mathcal{K}C L(lr)^2 +4( lr)^2+ 2( l\eta)^2 
	\end{equation}
	where t is a global iteration counter and $t=\sum_{k=1}^{\mathcal{K}-1} N_k$.
\end{lemma}
\begin{proof}
	From $x_t$ ( or $\tilde{x}^{(\mathcal{K})}$)to $x_{t+1}$ is one step of full gradient.
	\begin{align*}
	E[\|  x_{t+1} - \tilde{x}^0\|^2] &\leq 2E[\|  x_{t+1} -x_{t}\|^2]+ 2E[\|x_{t}-\tilde{x}^0\|^2]\\
	&= 2E[\eta^2\|\nabla f( x_t) \|^2]+ 2E[\|  \tilde{x}^{\mathcal{K}} - \tilde{x}^0\|^2]\\ 
	&\leq2( l\eta)^2 + 4\mathcal{K}C f_{thres} + 2\mathcal{K}C L(lr)^2 +4( lr)^2
	\end{align*}
	Hence, we obtain upper bound for $E[\| x_{t+1} - \tilde{x}^0 \|^2]$.
\end{proof}

	In the following part, we also want to obtain upper bound for $E[\| x_{t+1} - \tilde{x}^0 \|^2]$. The following part analysis is iteration based, where $t = \sum_{k=1}^{\mathcal{K}-1} N_k$. We use $x_0$ to denote $\tilde{x}^{0}$ without ambiguity.
\begin{lemma} \citep{nesterov2013introductory}
	[Taylor expansion bound for the gradient By Nesterov] For every twice differentiable function $f: \mathbb{R}^d \longrightarrow \mathbb{R}$ with $\rho-$Lipshitz Hessians, the following bound holds true:
	$$\|\nabla f( x ) - \nabla g( x )\| \leq \frac{\rho}{2}\|x - x_0\|^2,$$
	where 
	$$g( x ) := f( x_0) + ( x- x_0)^T\nabla f(x_0) + \frac{1}{2} ( x - x_0 )^T H ( x -x_0 ),$$
	where  $H = \nabla^2 f( x_0)$and $x$ is close enough to $x_0$.
\end{lemma}

Let's work on lower bound of $E[\|x_{t+1} - x_0 \|^2]$.
\begin{align*} 
x_{t+1} - x_0 &=  x_t - x_0 - \eta \nabla f( x_t) + \eta \xi_t \\ 
&=  x_t - x_0 - \eta \nabla g_t +  \eta (\nabla g_t - \nabla f( x_t) +  \xi_t)\\
&\downarrow replace \nabla g_t = \nabla f(x_0) +  H ( x -x_0 )\\
&= ( I - \eta H )( x_t - x_0 )+  \eta( \nabla g_t - \nabla f( x_t)- \nabla f( x_0) +  \xi_t )\\
&= ( I - \eta H )^t( x_1 - x_0 )+  \eta( \sum_{k=1}^{t}( I - \eta H)^{t-k}(\nabla g_k - \nabla f_k) \\&- \sum_{k=1}^{t}( I - \eta H)^{t-k} \nabla f( x_0) +  \sum_{k=1}^{t}( I - \eta H)^{t-k} \xi_k )\\	 
&= u_t + \eta( \delta_t + d_t +\zeta_t ),
\end{align*}

where\\
$u_t := ( I - \eta H )^t( x_1 - x_0 )$; \\
$\delta_t := \sum_{k=1}^{t}( I - \eta H)^{t-k}(\nabla g(x_k) - \nabla f(x_k))$;\\
$d_t := - \sum_{k=1}^{t}( I - \eta H)^{t-k} \nabla f( x_0)$;\\
$\zeta_t := \sum_{k=1}^{t}( I - \eta H)^{t-k} \xi_k $.

Then,
\begin{align*}
	E[ \| x_{t+1} - x_0 \|^2] &= E[\| u_t + \eta( \delta_t + d_t +\zeta_t )\|^2]\\
	& \geq E[\| u_t \|^2] +2\eta E[ u_t^T  \delta_t ]
	+ 2\eta E[ u_t^Td_t ]+  2\eta E[ u_t^T \zeta_t ]\\
	&\geq E[\| u_t \|^2] - 2\eta E[\| u_t \|]E[\| \delta_t \|]
	+ 2\eta E[ u_t]^Td_t + 2\eta E[ u_t ]^TE[ \zeta_t ].
\end{align*}

\begin{lemma}\label{removedependency}
	(Removing initial gradient dependency).
	\begin{equation}
	E[u_t]^T d_t\geq 0.
	\end{equation}
\end{lemma}
\begin{proof}
	$$E[u_t] = ( I - \eta H)^tE[x_1 - x_0] = -r( I - \eta H)^t \nabla f( x_0).$$
	Since $ I - \eta H \succeq 0$, 
	\begin{align*}
		E[u_t]^T d_t &= r(( I - \eta H)^t \nabla f( x_0))^T \sum_{k=1}^{t}( I - \eta H)^{t-k} \nabla f( x_0)\\
		&= r \sum_{k=1}^{t}\nabla f( x_0)^T( I - \eta H)^{2t-k} \nabla f( x_0) \\
		& \geq 0.
	\end{align*}
\end{proof}

\begin{lemma}\label{power}
	(Exponential Growing Power Iteration ). After the SCSG runs $t$ steps, it yields an exponentially growing lower bound on the expected squared norm of $u_t$ as follows:
	\begin{equation}
	E[\|u_t\|^2]\geq \tau r^2 \kappa^{2t}.
	\end{equation}
\end{lemma}

\begin{proof}
	Let $v$ represents the eigenvector corresponding to the most negative eignvalue $\lambda_{min}( H)$,\\
	$$	E[\|u_t\|^2] = 	E[\|v\|^2\|u_t\|^2] \geq E[(v^Tu_t)^2].$$
	Suppose $H = U^T\Sigma U$, then $( I -\eta H ) = U^T ( I - \eta\Sigma )U$. Further, we have:
	$$v^T( I - \eta H ) = v^T( 1 - \eta \lambda_{min}( H )) =  v^T( 1 + \eta \lambda).$$ 
	Finally, we get:
	\begin{align*}
		E[\|u_t\|^2] \geq r^2 ( 1 + \eta \lambda)^{2t}E[( v^T \nabla f_z( x_0))^2] \geq \tau r^2 (\kappa)^{2t}
	\end{align*}
	where $\kappa =  ( 1 + \eta \lambda)$ and $E[( v^T \nabla f_z( x_0))^2] \geq \tau $ is CNC-assumption.
\end{proof}
\begin{remark}\label{remark:1}
We know $\lambda \geq \epsilon_h$, due to negative curvature assumption.
\end{remark}

\begin{lemma}\label{upperbound}
 The norm of $u_t$ is determinstically bounded as
	\begin{equation}
	\|u_t\|\leq rl\kappa^t .
	\end{equation}
\end{lemma}
\begin{proof}
	\begin{align*}
	\|u_t\| &= \|( I - \eta H )^t( x_1 - x_0 ) \|\\
	&\leq \|( I - \eta H )\|^t\|( x_1 - x_0 ) \| \\
	&\leq ( 1 + \eta \lambda)^t r\|\nabla f_z(x_0 ) \| \\
	&\leq \kappa^t r l.\\
	\end{align*}
\end{proof}

\begin{lemma} \label{delta_bound}
	Assume $E[ f( x_{t})] - f( x_0) \geq -f_{thres} $ holds, then we can get
	$$ 	E[\| \delta_t \|] \leq  (\frac{2\rho C f_{thres}}{\eta^2 \lambda^2}  +\frac{\rho C L(lr)^2}{\eta^2 \lambda^2} + \frac{2 \rho l^2 r^2}{\eta \lambda} )\kappa^t. $$
\end{lemma}
\begin{proof}
	\begin{align*} 
	E[\| \delta_t \|] &= E[\|\sum_{k=1}^{t}( I - \eta H)^{t-k}(\nabla g(x_k) - \nabla f(x_k)) \|]\\
	&\leq \sum_{k=1}^{t}( 1+\eta \lambda)^{t-k} E[\|\nabla g(x_k) - \nabla f(x_k)\|] \\
	&\leq \frac{\rho}{2}\sum_{k=1}^{t}( \kappa)^{t-k} E[\| x_k - x_0\|^2] \\
	&\leq (\frac{2\rho C f_{thres}}{\eta^2 \lambda^2}  +\frac{\rho C L(lr)^2}{\eta^2 \lambda^2} + \frac{2 \rho l^2 r^2}{\eta \lambda} )\kappa^t.
	\end{align*}
\end{proof}

\begin{lemma}\label{distance_lower_bound}
	[Distance Lower Bound]
	Suppose that expectation of the decrease function value is lower-bounded:
	$$E[ f( x_{t})] - f( x_0) \geq -f_{thres}. $$
	Then, the expectation of the distance from the current iteration $x_{t+1}$ to $\tilde{x}^0$ is bounded as
	\begin{equation} \label{lowerbound4distiance}
	E[ \| x_{t+1} - x_0 \|^2] \geq \frac{1}{3} \tau r^2 \kappa^{2\frac{n}{b}(\mathcal{K}-1)}. 
	\end{equation}
	where $t=\sum_{k=1}^{\mathcal{K}-1} N_k$.
\end{lemma}
\begin{proof}
	
	We know that\\
	$ E[ u_t ]^TE[ \zeta_t ] = 0$; \\
	$E[\| u_t^Td_t \|]> 0$ by Lemma \ref{removedependency}; \\
	$E[\| u_t \|^2] \geq \tau r^2 \kappa^{2t}$ by Lemma \ref{power};\\
	$E[\| u_t \|] \leq rl \kappa^t $ by Lemma \ref{upperbound};\\
	$ 	E[\| \delta_t \|] \leq  (\frac{2\rho C f_{thres}}{\eta^2 \lambda^2}  +\frac{\rho C L(lr)^2}{\eta^2 \lambda^2} + \frac{2 \rho l^2 r^2}{\eta \lambda} )\kappa^t $ by Lemma \ref{delta_bound}.

	\begin{align*}
	E[ \| x_{t+1} - x_0 \|^2] &\geq  E[\| u_t \|^2] - 2\eta E[\| u_t \|]E[\| \delta_t \|]
	+ 2\eta E[ u_t]^Td_t + 2\eta E[ u_t ]^TE[ \zeta_t ]\\
	&\geq  \tau r^2 \kappa^{2t} - 2\eta rl \kappa^t ( (\frac{2\rho C f_{thres}}{\eta^2 \lambda^2}  +\frac{\rho C L(lr)^2}{\eta^2 \lambda^2}+ \frac{2 \rho l^2 r^2}{\eta \lambda} )\kappa^t)\\
	&\geq ( \tau r - ( \frac{4 l \rho C f_{thres}}{\eta \lambda^2} +\frac{2\rho C Ll^3r^2}{\eta \lambda^2}+ \frac{4\rho l^3 r^2}{\lambda} ) r \kappa^{2t}.\\
	\end{align*}
	
	In order to get contradiction on the upper bound and lower bound of $ E[ \| x_{t+1} - x_0 \|^2]$, we will first try to setup parameters for $f_{thres}$ and $r$ for lower bound:
	\begin{align*}
	\frac{4\rho l^3 r^2}{\lambda} &\leq \frac{1}{6} \tau r \Rightarrow r  \leq \frac{\tau \lambda}{24 \rho l^3}\xRightarrow{\lambda \geq \epsilon_h} r  \leq \frac{\tau \epsilon_h}{24 \rho l^3}\\
	\frac{2\rho C Ll^3r^2}{\eta \lambda^2} &\leq \frac{1}{6} \tau r \Rightarrow r  \leq \frac{\eta\tau \lambda^2}{12 \rho CL l^3}\xRightarrow{\lambda \geq \epsilon_h} r  \leq \frac{\eta\tau \epsilon_h^2}{12 \rho C L l^3}, \\
	\frac{4 l \rho C f_{thres}}{\eta \lambda^2} &\leq \frac{1}{3} \tau r \Rightarrow f_{thres} \leq \frac{\eta \tau r \lambda^2}{12 l \rho C} \xRightarrow{\lambda \geq \epsilon_h} f_{thres} \leq \frac{\eta \tau r \epsilon_h^2}{12 l \rho C}.
	\end{align*}
	where  $\lambda \geq \epsilon_h$ is due to remark \ref{remark:1}
	Therefore,
	\begin{align}
		&r  \leq min( \frac{\tau }{24 \rho l^3},  \frac{\eta\tau }{12 \rho C Ll^3})\epsilon_h^2 = C_2 \frac{\tau}{12 \rho l^3} \epsilon_h^2,\label{constraint_r} \\
		&f_{thres }  \leq \frac{\eta \tau r \epsilon_h^2}{12 l \rho C} \label{constraint_f},
	\end{align}
	where $C_2 = min( \frac{1 }{2},  \frac{\eta }{ C L})$.
	
	Given above requirements for parameter $f_{thres}$ and $r$, we obtain the new lower bound:
	\begin{align}
		E[ \| x_{t+1} - x_0 \|^2] &\geq E[\frac{1}{3} \tau r^2 \kappa^{2t}]\\
		&= E[\frac{1}{3} \tau r^2 \kappa^{2\sum_{k=1}^{\mathcal{K}-1}N_k}]\\
		&\geq\frac{1}{3} \tau r^2 \kappa^{2\sum_{k=1}^{\mathcal{K}-1}E[N_k]}\\
		&=\frac{1}{3} \tau r^2 \kappa^{2\frac{n}{b}(\mathcal{K}-1)}
	\end{align}
	where the last inequality is due to $N_k, k = 1, ..., \mathcal{K}-1$ are independent with each other and Jensen inequality.

\end{proof}

\begin{lemma}\label{eigenvalue}
	If the Hessian matrix at $x_0$ has a small negative eigenvalue, i.e. 
	$ \lambda_{min}( \nabla f^2( x_0 )) \leq -\epsilon_h$
	and $  \mathcal{K}_{thres} \geq \frac{C_1}{\eta \epsilon_h} \log( \frac{12 l}{\tau r \epsilon_h}) $, where $C_1 $ is a sufficient large constant,
	then there exists a state $\mathcal{K} \leq \mathcal{K}_{thres}$ such that expectation of the function value decrease as 
	$$ f( \tilde{x}^0) - E[f( \tilde{x}^{\mathcal{K}})]  \geq f_{thres}.  $$
	
\end{lemma}
\begin{proof}

	By Lemma \ref{distance_lower_bound}, we can see that the distance between $x_{t+1}$ and $x_0$ is lower bounded by an exponential term in (\ref{lowerbound4distiance})
		\begin{equation*} 
	E[ \| x_{t+1} - x_0 \|^2] \geq \frac{1}{3} \tau r^2 \kappa^{2\frac{n}{b}(\mathcal{K}-1)}. 
	\end{equation*}
	However, in Lemma \ref{distance_Upper_bound_t+1}, we also obtain an upper bound for this distance in (\ref{upperbound4distiance_t+1}),
		\begin{equation*} 
	E[\|  x_{t+1} - \tilde{x}^0\|^2] \leq 4\mathcal{K}C f_{thres} + 2\mathcal{K}C L(lr)^2 +4( lr)^2+ 2( l\eta)^2 
	\end{equation*}
	which is a linear function in terms of iteration $\mathcal{K}$. 
	
	First, let 
	\begin{align} \label{eq:1}
	   4\mathcal{K}C f_{thres} + 2\mathcal{K}C L(lr)^2 +4( lr)^2+ 2( l\eta)^2 = 0
	\end{align}

	Put all constrains (\ref{constraint_r}), (\ref{constraint_f}) and $\tau \leq l^2$ into equality (\ref{eq:1}), we get a new linear function in terms of $\mathcal{K}$:
		\begin{align} \label{eq:2}
	   (\frac{ \eta^2 }{36 \rho^2 C L  } + \frac{ \eta^2}{72 \rho^2 C L^2 })\epsilon_h^4\mathcal{K}+4( lr)^2+ 2( l\eta)^2 = 0
	\end{align}
	Then based on $\epsilon_h^4 \leq 1 $, we get : 
	\begin{align} \label{eq:2}
	   (\frac{ \eta^2 }{36 \rho^2 C L  } + \frac{ \eta^2}{72 \rho^2 C L^2 })\mathcal{K}+4( lr)^2+ 2( l\eta)^2 = 0
	\end{align}
	The point of intersection of  $$y = (\frac{ \eta^2 }{36 \rho^2 C L  } + \frac{ \eta^2}{72 \rho^2 C L^2 })\mathcal{K}+4( lr)^2+ 2( l\eta)^2 ~~\&~~ y=\frac{1}{3} \tau r^2 \kappa^{2\frac{n}{b}(\mathcal{K}-1)} $$ is larger than the point of intersection of $$ y = 4\mathcal{K}C f_{thres} + 2\mathcal{K}C L(lr)^2 +4( lr)^2+ 2( l\eta)^2 ~~\&~~ y=\frac{1}{3} \tau r^2 \kappa^{2\frac{n}{b}(\mathcal{K}-1)} $$
	
	Therefore, we only need to find the point of intersection $\mathcal{K}_{point}$ of  $$y = (\frac{ \eta^2 }{36 \rho^2 C L  } + \frac{ \eta^2}{72 \rho^2 C L^2 })\mathcal{K}+4( lr)^2+ 2( l\eta)^2 = 0 ~~\&~~  y=\frac{1}{3} \tau r^2 \kappa^{2\frac{n}{b}(\mathcal{K}-1)} $$  and let $\mathcal{K}_{thresh} \geq \mathcal{K}_{point}$.

Let $$f( \mathcal{K}) = \frac{1}{3} \tau r^2 (1+ \eta \epsilon_h)^{2\frac{n}{b}(\mathcal{K}-1)} - (\frac{ \eta^2 }{36 \rho^2 C L  } + \frac{ \eta^2}{72 \rho^2 C L^2 })\mathcal{K}-4( lr)^2- 2( l\eta)^2$$
Then we calculate the gradient of $f( \mathcal{K})$ and let it be 0:
$$f'( \mathcal{K}) = \frac{2}{3} \frac{n}{b} \tau r^2 (1+ \eta \epsilon_h)^{2\frac{n}{b}(\mathcal{K}-1)} log( (1+ \eta \epsilon_h) - (\frac{ \eta^2 }{36 \rho^2 C L  } + \frac{ \eta^2}{72 \rho^2 C L^2 })=0 $$
\begin{align*}
    \mathcal{K^*} &= \frac{b}{2n}\frac{log( \frac{3b}{2n\tau r^2 log( (1+ \eta \epsilon_h) }(\frac{ \eta^2 }{36 \rho^2 C L  } + \frac{ \eta^2}{72 \rho^2 C L^2 }))}{log( 1 + \eta \epsilon_h)} + 1\\
    & \leq \frac{b}{2n \eta \epsilon_h}log( \frac{3b}{2n\tau r^2 \eta \epsilon_h  }(\frac{ \eta^2 }{36 \rho^2 C L  } + \frac{ \eta^2}{72 \rho^2 C L^2 })) + 1
\end{align*}
here we use $log (1+a) \geq a$ if $a \in (-1,1]$.

Based on property of function $f( \mathcal{K})$, we know :
\begin{align*}
    &\mathcal{K}_{point} \leq 2\mathcal{K^*}  + \frac{4( lr)^2+ 2( l\eta)^2}{(\frac{ \eta^2 }{36 \rho^2 C L  } + \frac{ \eta^2}{72 \rho^2 C L^2 })}\\
    &=\frac{b}{n \eta \epsilon_h}log( \frac{3b}{2n\tau r^2 \eta \epsilon_h  }(\frac{ \eta^2 }{36 \rho^2 C L  } + \frac{ \eta^2}{72 \rho^2 C L^2 })) + \frac{4( lr)^2+ 2( l\eta)^2}{(\frac{ \eta^2 }{36 \rho^2 C L  } + \frac{ \eta^2}{72 \rho^2 C L^2 })} +  2\\
    & =\frac{b}{n \eta \epsilon_h}(log( \frac{1}{\epsilon_h^5}) + log( \frac{3*144\rho^2l^6b}{2nC_2^2\tau^3\eta  }(\frac{ \eta^2 }{36 \rho^2 C L  } + \frac{ \eta^2}{72 \rho^2 C L^2 })) + \frac{4( lr)^2+ 2( l\eta)^2}{(\frac{4 \eta }{C L \rho^2 } + \frac{2 \eta^2}{144 C L^2 \rho^2})} + 2\\
    &= \frac{b}{n \eta \epsilon_h}log( \frac{1}{\epsilon_h^5})(1 + \frac{1}{log( \frac{1}{\epsilon_h^5})}log( \frac{3*144\rho^2l^6b}{2nC_2^2\tau^3\eta  }(\frac{ \eta^2 }{36 \rho^2 C L  } + \frac{ \eta^2}{72 \rho^2 C L^2 })) + \frac{\epsilon_h}{log( \frac{1}{\epsilon_h^5})}(\frac{4( lr)^2+ 2( l\eta)^2}{(\frac{ \eta^2 }{36 \rho^2 C L  } + \frac{ \eta^2}{72 \rho^2 C L^2 })} +  2)\\
    &\leq \frac{b}{n \eta \epsilon_h}log( \frac{1}{\epsilon_h^5})(1 + \frac{1}{log( \frac{1}{(0.1)^5})}log( \frac{3*144\rho^2l^6b}{2nC_2^2\tau^3\eta  }(\frac{ \eta^2 }{36 \rho^2 C L  } + \frac{ \eta^2}{72 \rho^2 C L^2 })) + \frac{1}{log( \frac{1}{(0.1)^5})}(\frac{4( lr)^2+ 2( l\eta)^2}{(\frac{4 \eta }{C L \rho^2 } + \frac{2 \eta^2}{144 C L^2 \rho^2})} +  2)\\
     &\leq \frac{b}{n \eta \epsilon_h}log( \frac{1}{\epsilon_h^5})(1 + \frac{1}{log( \frac{1}{(0.1)^5})}log( \frac{3*144\rho^2l^6b}{2nC_2^2\tau^3\eta  }(\frac{ \eta^2 }{36 \rho^2 C L  } + \frac{ \eta^2}{72 \rho^2 C L^2 }))\\& + \frac{1}{log( \frac{1}{(0.1)^5})}(\frac{4( l)^2C_2 \frac{\tau}{12 \rho l^3}+ 2( l\eta)^2}{(\frac{ \eta^2 }{36 \rho^2 C L  } + \frac{ \eta^2}{72 \rho^2 C L^2 })} +  2)\\
      &\leq C_1 \frac{b}{n \eta \epsilon_h}log( \frac{1}{\epsilon_h^5})
\end{align*}
where $C_1 \geq (1 + \frac{1}{log( \frac{1}{(0.1)^5})}log( \frac{3*144\rho^2l^6b}{2nC_2^2\tau^3\eta  }(\frac{ \eta^2 }{36 \rho^2 C L  } + \frac{ \eta^2}{72 \rho^2 C L^2 })) + \frac{1}{log( \frac{1}{(0.1)^5})}(\frac{4( l)^2C_2 \frac{\tau}{12 \rho l^3}+ 2( l\eta)^2}{(\frac{4 \eta }{C L \rho^2 } + \frac{2 \eta^2}{144 C L^2 \rho^2})} +  2)$

In summary,  the lower bound will be eventually larger than the upper bound with the increase of epoch $\mathbb{K}$ and let $ \mathcal{K}_{thres} \geq \frac{C_1}{\eta \epsilon_h} \frac{n}{b}\log( \frac{1}{\epsilon_h}) $, where $C_1 $ is a sufficient large constant and independent with $\epsilon_h$. The contradiction is introduced and the assumption about decrease of function value is broken.

\end{proof}

\textbf{Proof of Convergence }
In order to cooperate with constraints derived in previous cases, $g_{thres}$ should be carefully selected. We argue that constants $\gamma$ and $C_1$ exist such that 
	
	\begin{align*}
		 \frac{10l^2\gamma^2}{L \delta}(\frac{b}{n})^{\frac{1}{3}} &\leq \frac{n}{b}\frac{\eta^2 \epsilon_h^3 \tau r}{C_1 12 l \rho C} log^{-1}( \frac{1}{\epsilon_h}) 
		\leq \frac{\gamma}{5L} ( \frac{n}{b})^\frac{1}{3} \epsilon_g^2, 
	\end{align*}
	
	where $\gamma$ is a small constant and $C_1$ is a large constant.
	
	With above argument about the existence of $\gamma, C_1$, let $g_{thres} = \frac{\eta^2 \epsilon_h^3 \tau r}{C_1 12 l \rho C} \log^{-1}( \frac{12l}{\tau r \epsilon_h})$ and we are ready to prove the main theorem under probabilistic setting.

We define event $A_t$ as 
$$ A_t := \{\|\nabla f( \tilde{x}^{t} ) \| \geq \epsilon_g ~or~ \lambda_{min}( \nabla^2 f( \tilde{x}^{t}  )) \leq -\epsilon_h\}.$$
Let $R$ be a random variable represented the ratio of second-order stationary points visited at the end of each epoch in past T epochs by our algorithm. Hence:
$R = \frac{1}{T}\sum_{t =1 }^{T} \mathbf{I} (A_t^c),$
where $\mathbf{I}(\cdot)$ is an indicator function. Further, let $P_t$ be the probability of event $A_t$ and $1-P_t$ be the probability of its complement $A_t^c$. Our goal will be 
$$E( R) = \frac{1}{T} \sum_{t=1}^{T}( 1 - P_t)$$
in high probability, which means:
$$E( R) = \frac{1}{T} \sum_{t=1}^{T}( 1 - P_t) \geq 1- \delta, $$
or
$$\frac{1}{T}\sum_{t=1}^{T} P_t \leq \delta.$$

Estimating the probabilities $P_t$ for $t \in \{1, ..., T\}$ are very difficult. However, we can bound them together. Our algorithm obtains the guaranteed function value decrease per epoch in large gradient and large negative curvature regimes:
$$E[ f( \tilde{x}^{t} ) - f( \tilde{x}^{t-1} )| A_t] \leq -g_{thres},$$
and the increase of function value around second-order stationary points is controlled by our choice of parameters:
$$E[f( \tilde{x}^{t} ) - f( \tilde{x}^{t-1} )| A_t^c] \leq \delta g_{thres}/2.$$
Then we can obtain expected change of function values:
$$E[ f( \tilde{x}^{t} ) - f(\tilde{x}^{t-1} )] \leq \delta (g_{thres}/2 )( 1 - P_t) - g_{thres}P_t.$$
Summing the above inequality over the $T$ and reorganizing, we obtain:
$$\frac{1}{T}\sum_{t=1}^{T} P_t \leq \frac{f( x_0) - f^*}{Tg_{thres}} + \frac{\delta}{2} \leq \delta.$$
Therefore, 
$$T = \frac{2( f( x_0) - f^*)}{\delta g_{thres}} 
= \frac{b}{n} \frac{288 C C_1\rho^2 l^4 ( f( x_0) - f^*)}{C_2\delta \eta^2 \tau^2 \epsilon_h^5} \log( \frac{1}{\epsilon_h})
= \frac{b}{n} \frac{288 C C_1 l^4 ( f( x_0) - f^*)}{C_2\delta \eta^2 \tau^2 \epsilon^2} \log( \frac{1}{ \sqrt{\rho}\epsilon^{\frac{2}{5}}}).$$

\subsection{Proof of Computational Complexity}
Considering constants $C$, $C_1$, $C_2$ and $\gamma$ in Theorem 3.1 are independent of $\epsilon$, the convergence rate will be $E[\sum_{k=1}^{T} N_k]=\frac{n}{b}T =\tilde{O}( \epsilon^{-2} log( 1/\epsilon))$. If we further require $C_4 = \frac{n}{b}$ to be a constant independent with $n$, then the number of epochs $T$ only depends on $C_4=\frac{n}{b}$. Therefore, the number of IFO calls can be derived by $E[\sum_{k=1}^{T} (n + b N_k )]=2nT=O( n\epsilon^{-2} log( 1/\epsilon))$.

\section{Proofs for the proposed general framework}
As the same situation in Neon and Neon2, generic framework also depends on how to check the $\epsilon_g$-first-order condition (Line 3). In practice, we can  use some approximation approaches to check the $\epsilon_g$-first-order condition. For practical consideration, we can make sample size $|S|$ large enough and use $\nabla f_{S}(x_{t-1}^k)$ to check $\epsilon_g$-first-order condition. When we say sample size large enough, it means that $\| \nabla f_{S}(x_{t-1}^k) -  \nabla f(x_{t-1}^k)\| \leq \epsilon/2$ with high probability $1-\delta$, hence $\| \nabla f_{S}(x_{t-1}^k) \| \leq \epsilon/2$ can guarantee that  $\| \nabla f(x_{t-1}^k) \| \leq \epsilon$ with high chance. We say the $\epsilon_g$-first-order stationary point is found and the algorithm is ready for SGD jump step. The sample size $|S|$  can be determined by the following lemma which was first established in \citep{ghadimi2016mini}:
\begin{lemma}\label{prob}
	Suppose function $f$ is gradient-bounded . Let $\nabla f_{S} ( x )= \frac{1}{|S|} \sum_{z\in S} \nabla f_z( x) $, for any $\epsilon, \delta \in (0, 1)$, $x\in\mathbb{R}^d$, when $|S| \geq \frac{2l^2( 1+ log( 1/\delta))}{\epsilon^2}$, we have $Pr( \| \nabla f_S( x) - \nabla f( x ) \| \leq \epsilon) \geq 1 - \delta$.
\end{lemma}
\subsection{Proof of Theorem 4.3}
Lemma 4.1 guarantees that the decrease of function value per epoch in expectation around strict saddle points is $D_e$. The assumption about algorithm $ \mathcal{A} $ in Theorem 4.2 makes sure that $ \mathcal{A} $ will drop function value in $D_{\mathcal{A}( \epsilon_g)}$ per epoch in expectation when $\|\nabla f( x)\| > \epsilon_g$ and won't increase function value too much in expectation when  $\|\nabla f( x) \|\leq \epsilon_g$ and $\lambda_{min}( \nabla^2 f( x) ) \leq  \epsilon_h$. Then let $g_{thres} = min(D_e, D_{\mathcal{A}( \epsilon_g)})$ and follow previous argument for main Theorem , we can obtain the desired result.

\section{A concrete case for algorithm $\mathcal{A}$ }
Let  algorithm $\mathcal{A}$ be one epoch of SCSG with batch size $B_1$ and $b_1$, then by Theorem 3.1, we can  obtain that
$$ E[\|\nabla f ( y)\|^2] \leq \frac{5 b_1}{\eta B_1} E[ f( x) - E( y)] + \frac{6\mathbf{I}( B_1 < n )}{B_1} l^2$$                                                                                                 
where $\mathbf{I}(\cdot )$ is an indicator function.
\begin{algorithm}[htb]
	\caption {SCSG-epoch}
	\label{alg:SCSG_epoch} 
	\begin{algorithmic}
		\Require $x$, $\eta$, $B_1$, $b_1$
		\State Randomly pick $I_B \in \{1,\dots,n\}$ where $|I_B| = B_1$  
		\State $\tilde{\mu} =  \nabla f_{I_B}(x)$
		\State Generate $N \sim $ Geom( $B_1/B_1-b_1$) 
		\For{$t=1,2,\dots, N$}
		\State Randomly pick $I_b \subset \{1,...,n\}$ and update weight, where $| I_b | = b_1$
		\State $x_t = x_{t-1} - \eta( \nabla f_{I_b}(x_{t-1}) - \nabla f_{I_b}( x) + \tilde{\mu} )$
		\EndFor
		\State return $ y = x_N$
	\end{algorithmic}
\end{algorithm}

When $\|\nabla f ( y)\|  > \epsilon_g$, 
\begin{align*}
E[f( x) - E( y)] &\geq \frac{\eta B_1}{b_1} ( \epsilon_g^2 -  \frac{6\mathbf{I}( B_1 < n )}{B_1} l^2) \\
&\geq  \frac{\eta B_1}{2b_1}  \epsilon_g^2 
\end{align*}
where $B_1 \geq \frac{12l^2}{\epsilon_g^2} $. Therefore, $D_{\mathcal{A}( \epsilon_g)} = \frac{\eta B_1 }{2 b_1}  \epsilon_g^2$.

When $\|\nabla f ( y)\|  \leq \epsilon_g$, 
\begin{align*}
E( f( y) - E( x)) &\leq \frac{\eta B_1}{b_1} (  \frac{6\mathbf{I}( B_1 < n )}{B_1} l^2) \\
&\leq  \frac{\eta }{2b_1}  l^2.
\end{align*}
Then, the expected increase of function value is  $\frac{\eta }{2b_1}  l^2$. Let
$$\frac{\eta }{2b_1}  l^2 \leq min(D_e, D_{\mathcal{A}( \epsilon_g)}), $$ 
$$b_1 \geq \frac{\eta l^2}{2} max(\frac{1}{D_e}, \frac{1}{D_{\mathcal{A}( \epsilon_g)}}).$$

Therefore, when $b_1 \geq  \frac{\eta l^2}{2} max(\frac{1}{D_e} ,\frac{1}{D_{\mathcal{A}( \epsilon_g)}})$ and $B_1 \geq max( \frac{12l^2}{\epsilon_g^2} ,b_1)$, SCSG-epoch satisfies the assumption of algorithm $\mathcal{A}$.


\section{The anisotropic noise in SGD}

In this section, we intend to justify the CNC-assumption. In non-convex optimization, the updating formula for SGD is:
$$x_{t+1} = x_{t} - r\nabla f_z( x_{t}),$$
where $r$ is the stepsize, $z$ is randomly sampled from $\{1,2,\dots,n\}.$

Then we know the covariance matrix of stochastic gradient $\nabla f_z(x )$:
\begin{align} 
Var(\nabla f_z( x ))
& = \frac{1}{n} \sum_{z= 1}^{n}  ( \nabla f_z( x ) - \nabla f( x ) ) (  \nabla f_z( x ) - \nabla f( x ) )^T \label{cov_fisher}\\ 
& = \frac{1}{n} \sum_{z= 1}^{n} \nabla f_z( x ) (\nabla f_z( x ))^T - \nabla f( x ) (\nabla f( x ) )^T \nonumber\\
&\Downarrow \nonumber\\
&around \;\; saddle\;\; points \Longrightarrow \|\nabla f( x ) \|\leq \epsilon	\nonumber\\
&\Downarrow \nonumber\\
&\approx\frac{1}{n} \sum_{z= 1}^{n} \nabla f_z( x ) (\nabla f_z( x ))^T \nonumber\\
&=F_E( x ), \nonumber
\end{align}
where $F_E(x)$ is empirical fisher information matrix. 
\begin{lemma} \label{fisher_hessian}
	Assume $ P( z | x) \propto exp^{ -f_z( x)} $, then
	
	$$E_{P(z|x)}[\nabla^2 \log( P( z|x))] = - F,$$
	where $F = E_{P( z|x)}[ \nabla f_z( x) \nabla f_z( x)^T]$ is Fisher information matrix.

\end{lemma}
\begin{proof}
	\begin{align*}
	\nabla^2 \log( P( z|x))&=\frac{\nabla^2 P( z|x) P ( z| x ) - \nabla P( z|x) \nabla P( z|x)^T }{P( z|x) P ( z| x )} \\
	&=\frac{\nabla^2 P( z|x) P ( z| x )}{P( z|x) P ( z| x )^T} -\frac{ \nabla P( z|x) \nabla P( z|x)^T }{P( z|x) P ( z| x )}\\
	&=\frac{\nabla^2 P( z|x) }{P( z|x) } -\frac{ \nabla P( z|x)  }{P( z|x)} (\frac{ \nabla P( z|x)  }{P( z|x)})^T.\\
	\end{align*}
	
	Therefore:
	\begin{align*}
	&E_{P(z|x)}[\nabla^2 \log( P( z|x))]\\
	&=E_{P(z|x)}[\frac{\nabla^2 P( z|x) }{P( z|x) }] -E_{P(z|x)}[\frac{ \nabla P( z|x)  }{P( z|x)} (\frac{ \nabla P( z|x)  }{P( z|x)})^T]\\
	&= -E_{P(z|x)}[\nabla log(P( z|x) ) (\nabla log(P( z|x) ))^T]\\
	&=-E_{P(z|x)}[\nabla f_z( x) f_z( x)^T]\\
	&=-F
	\end{align*}
	where $F$ is Fisher information matrix.
\end{proof}
	Around saddle points, the variance-covariance matrix of the stochastic gradient $\nabla f_z( x)$ is approximately equal to the empirical Fisher's information matrix (see Appendix). In practice, under the assumptions that the sample size $n$ is large enough and the current $x$ is not too far away from the ground truth parameter, Fisher's information matrix $E_{P(z|x)}[\nabla^2 \log( P( z|x))]$ approximately equals to the Hessian matrix $H$ of $f(\cdot)$ at $x$. Thus, the intrinsic noise in SGD is naturally embedded into Hessian information. This explains the anisotropic behavior of the noise in SGD.
\end{document}